\documentclass[11pt]{amsart}
\usepackage{mathpazo}
\usepackage{latexsym}   
\usepackage{amssymb}    
\usepackage[colorlinks]{hyperref}
\usepackage[margin=1in]{geometry}
\usepackage{booktabs}
\usepackage{pgfplots}
\pgfplotsset{compat=1.3}

\newtheorem{theorem}{Theorem}[section]
\newtheorem{conjecture}[theorem]{Conjecture}
\newtheorem{lemma}[theorem]{Lemma}
\newtheorem{proposition}[theorem]{Proposition}
\newtheorem{corollary}[theorem]{Corollary}

\theoremstyle{definition}

\newenvironment{alphabetize}{\begin{enumerate}

}{\end{enumerate}}

\DeclareMathOperator{\aut}{Aut}
\DeclareMathOperator{\cl}{Cl}
\DeclareMathOperator{\cond}{{\mathfrak f}}
\DeclareMathOperator{\End}{End}
\DeclareMathOperator{\frob}{Fr}
\DeclareMathOperator{\gal}{Gal}
\DeclareMathOperator{\gl}{GL}
\DeclareMathOperator{\Hom}{Hom}
\DeclareMathOperator{\isom}{Isom}
\DeclareMathOperator{\jac}{Jac}
\DeclareMathOperator{\mat}{Mat}
\DeclareMathOperator{\norm}{\caln}
\DeclareMathOperator{\relcond}{\cond_\rel}
\DeclareMathOperator{\tr}{tr}

\newcommand{\cala}{{\mathcal A}}
\newcommand{\cali}{{\mathcal I}}
\newcommand{\calm}{{\mathcal M}}
\newcommand{\caln}{{\mathcal N}}
\newcommand{\calo}{{\mathcal O}}
\newcommand{\cals}{{\mathcal S}}

\newcommand{\affine}{{\mathbf A}}
\newcommand{\bolde}{{\mathbf E}}

\newcommand{\aalpha}{{\pmb\alpha}}
\newcommand{\mmu}{{\pmb\mu}}

\newcommand{\complexes}{{\mathbb C}}
\newcommand{\ff}{{\mathbb F}}
\newcommand{\rat}{{\mathbb Q}}
\newcommand{\integ}{{\mathbb Z}}

\newcommand{\ida}{{\mathfrak a}}
\newcommand{\idp}{{\mathfrak p}}

\newcommand{\gsplit}{\textup{geom.~split}}
\newcommand{\rel}{\textup{rel}}
\newcommand{\ssplit}{\textup{split}}
\newcommand{\supersing}{{\textup{ss}}}

\newcommand{\abs}[1]{{\lvert#1\rvert}}
\newcommand{\bigfloor}[1]{{\left\lfloor #1 \right\rfloor}}
\newcommand{\ceil}[1]{{\lceil #1 \rceil}}
\newcommand{\floor}[1]{{\lfloor #1 \rfloor}}
\newcommand{\oneover}[1]{\frac{1}{#1}}
\newcommand{\quadres}[2]{\left(\frac{#1}{#2}\right)}
\newcommand{\st}[1]{\{#1\}}

\newcommand{\comp}{\circ}
\newcommand{\inv}{^{-1}}
\newcommand{\iso}{\cong}
\newcommand{\units}{^\times}
\newcommand{\fq}{\ff_q}
\newcommand{\fpbar}{\bar\ff_p}
\newcommand{\fqbar}{\bar\ff_q}

\renewcommand{\div}{\mid}
\renewcommand{\nsim}{\not\sim}
\renewcommand{\tilde}{\widetilde}

\newcommand{\mybar}[1]{
  \mathchoice
  {#1\llap{$\overline{\phantom{\displaystyle\rm#1}}$}}
  {#1\llap{$\overline{\phantom{\textstyle\rm#1}}$}}
  {#1\llap{$\overline{\phantom{\scriptstyle\rm#1}}$}}
  {#1\llap{$\overline{\phantom{\scriptscriptstyle\rm#1}}$}}
}  
\renewcommand{\bar}{\mybar}

\newcommand{\pibar}{{\overline{\pi}}}

\begin{document}

\title[Abelian surfaces and reductions of curves]
      {Split abelian surfaces over finite fields\\ and reductions of genus-$2$ curves}

\author{Jeffrey D. Achter}
\email{j.achter@colostate.edu}
\address{Department of Mathematics, Colorado State University, Fort Collins, CO 80523} 
\urladdr{\url{http://www.math.colostate.edu/~achter}}

\author{Everett W. Howe}
\email{however@alumni.caltech.edu}
\address{Center for Communications Research, 4320 Westerra Court, San Diego, CA 92121-1967}
\urladdr{\url{http://www.alumni.caltech.edu/~however}}

\thanks{The first author was partially supported by  grants from the
Simons Foundation (204164) and the NSA (H98230-14-1-0161 and H98230-15-1-0247)}

\dedicatory{Dedicated to the memory of Professor Tom M.~Apostol}

\date{4 October 2016}



\begin{abstract}
For prime powers $q$, let $\ssplit(q)$ denote the probability that a
randomly-chosen principally-polarized abelian surface over the finite field 
$\fq$ is not simple.  We show that there are positive constants $c_1$ and $c_2$
such that for all $q$,
\[
c_1 (\log q)^{-3}(\log \log q)^{-4} < \ssplit(q)\sqrt{q} < c_2 (\log q)^4(\log \log q)^2,
\]
and we obtain better estimates under the assumption of the generalized Riemann
hypothesis.  If $A$ is a principally-polarized abelian surface over a number
field $K$, let $\pi_\ssplit(A/K, z)$ denote the number of prime ideals $\idp$ 
of $K$ of norm at most $z$ such that $A$ has good reduction at $\idp$ and
$A_\idp$ is not simple. We conjecture that for sufficiently general $A$, the
counting function $\pi_\ssplit(A/K, z)$ grows like $\sqrt{z}/\log z$.  We 
indicate why our theorem on the rate of growth of $\ssplit(q)$ gives us reason
to hope that our conjecture is true.
\end{abstract}

\maketitle

\setcounter{tocdepth}{1}
\tableofcontents

\section{Introduction}
\label{secintro}

Let $A/K$ be a principally-polarized absolutely simple abelian variety over a
number field.  Murty and Patankar have conjectured 
\cite{Murty2005,MurtyPatankar2008} that if the absolute endomorphism ring of 
$A$ is commutative, then the reduction $A_\idp$ is simple for almost all primes 
$\idp$ of $\calo_K$.  (See \cite{Achter2009,Zywina2014} for work on this 
conjecture.)  Given this, it makes sense to try to quantify the (conjecturally
density zero) set of primes of good reduction for which $A_\idp$ is 
\emph{split}; that is, for which $A_\idp$ is isogenous to a product of abelian
varieties of smaller dimension. Specifically, define the counting function
\[
\pi_\ssplit(A/K, z) = \#\st{\idp : \norm(\idp)\le z
  \text{ and $A_\idp$ is split}}.
\]

Some upper bounds for the rate of growth of this function are available. For 
instance, a special case of \cite[Thm.~B, p.~42]{Achter2012} states that if the
image of the $\ell$-adic Galois representation attached to the $g$-dimensional
abelian variety $A$ is the full group of symplectic similitudes, then 
\begin{alignat}{2}
\notag
\pi_\ssplit(A/K,z) &\ll \frac{z (\log\log z)^{1+1/3(2g^2+g+1)}}
                             {(\log z)^{1+1/6(2g^2+g+1)}}
     &\text{\qquad for all $z\ge 3$;}\\
\intertext{if one is willing to assume a generalized Riemann
           hypothesis, one can further show that }
\label{eqzm}
\pi_\ssplit(A/K,z) &\ll z^{1-\oneover{4g^2+3g+4}}(\log z)^{\frac{2}{4g^2+3g+4}}
     &\text{for all $z\ge 3$.}
\end{alignat}
However, there is no reason to believe that even~\eqref{eqzm} does a very good 
job of capturing the actual behavior of the function $\pi_\ssplit(A/K,z)$.  The
purpose of the present paper is to explain and support the following hope.

\begin{conjecture}
\label{conjabsurf}
Let $A/K$ be a principally-polarizable abelian surface with absolute 
endomorphism ring $\End_{\bar K} A \cong \integ$.  Then there is a constant 
$C_A > 0$ such that
\[
\pi_\ssplit(A/K,z) \sim C_A \frac{\sqrt z}{\log z} \text{\qquad as $z\to\infty$}.
\]
\end{conjecture}

This statement bears some resemblance to the Lang--Trotter 
conjecture~\cite{LangTrotter1976}, whose enunciation we briefly recall.  Let
$E/\rat$ be an elliptic curve with $\End_{\bar\rat} E \iso \integ$, and fix a 
nonzero integer~$a$. Let $\pi(E,a,z)$ be the number of primes $p<z$ such that 
$E_p(\ff_p)-(p+1) = a$.  Then Lang and Trotter conjecture that 
$\pi(E,a,z) \sim C_{E,a} \sqrt z/\log z$ as $z\to\infty$, for some constant
$C_{E,a}$.  They also give a conjectural formula for the constant $C_{E,a}$,
but we shall ignore such finer information here.

In Section \ref{seclangtrotter}, we review a framework under which one might 
expect such counting functions to grow like $\sqrt z/\log z$. Roughly speaking,
the philosophy of Section~\ref{seclangtrotter} suggests that 
Conjecture~\ref{conjabsurf} should hold if the probability that a 
randomly-chosen principally-polarized abelian surface over $\fq$ is split 
varies like $q^{-1/2}$. The bulk of our paper is taken up with a proof of a 
theorem which says that, up to factors of $\log q$, this is indeed the case.

For every positive integer $n$ we let $\cala_n$ denote the moduli stack of
principally-polarized $n$-dimensional abelian varieties, so that for every 
field $K$ the objects of $\cala_n(K)$ are the $K$-isomorphism classes of such
principally-polarized varieties over~$K$. For every $n$ and $K$ we also let
$\cala_{n,\ssplit}(K)$ denote the subset of $\cala_n(K)$ consisting of the
principally-polarized abelian varieties $(A,\lambda)$ for which $A$ is not 
simple over~$K$. This is perhaps an abuse of notation, because there is no
geometrically-defined substack $\cala_{n,\ssplit}$ giving rise to the sets 
$\cala_{n,\ssplit}(K)$; our definition of `split' is sensitive to the field of
definition.

\begin{theorem}
\label{thmain}
We have
\[
\oneover{(\log q)^3 (\log\log q)^4}
\ll
\frac{\#\cala_{2,\ssplit}(\fq)}{q^{5/2}}
\ll
(\log q)^4 (\log\log q)^2 \text{\qquad for all $q$.}
 \]
If the generalized Riemann hypothesis is true, we have
\[
\oneover{(\log q) (\log\log q)^6} 
\ll
\frac{\#\cala_{2,\ssplit}(\fq)}{q^{5/2}} 
\ll
(\log q)^2(\log\log q)^4\text{\qquad for all $q$.}
\]
\end{theorem}

Since $\cala_2$ is irreducible of dimension $3$, Theorem~\ref{thmain} implies
that, up to logarithmic factors, the chance that a randomly chosen 
principally-polarized abelian surface over $\fq$ is split varies 
like~$q^{-1/2}$.

The paper closes by presenting some numerical data, including evidence in favor
of Conjecture~\ref{conjabsurf}.  We also indicate what we believe to be true
when one considers varieties that are geometrically split, and not just split
over the base field.

After the first-named author gave a preliminary report on this work, including 
some data obtained using \texttt{sage}, William Stein suggested contacting 
Andrew Sutherland for help with more extensive calculations.  Sutherland
provided us with the program \texttt{smalljac}~\cite{KedlayaSutherland2008},
which we ran on our own computers to obtain data on the mod-$p$ reductions 
of the curve $y^2 = x^5 + x + 6$ over~$\rat$; later, Sutherland kindly used
his own computers, running a program based on the algorithm 
in~\cite{HarveySutherland2014}, to provide us with reduction data for this 
curve for all primes up to $2^{30}$. It is a pleasure to acknowledge
Sutherland's assistance.  The data presented in Sections~\ref{ssec-data-relcond}
and~\ref{ssec-data-probsplit} was obtained using \texttt{gp} and \textsc{Magma}.

As we were writing up the various asymptotic estimates of number-theoretic 
functions that appear in this paper, the second-named author thought frequently
of Professor Tom M.~Apostol, in whose undergraduate Caltech course Math 160 he
first became familiar with such computations. Not long after we completed this
paper, Apostol passed away. We dedicate this work to his memory.

\subsection*{Notation and conventions}
If $Z$ is a set of real numbers and $f$ and $g$ are real-valued functions 
on~$Z$, we use the Vinogradov notation
\[
f(z) \ll g(z) \text{\qquad for $z \in Z$}
\]
to mean that there is a constant $C$ such that $\abs{f(z)} \le C \abs{g(z)}$ 
for all $z \in Z$. If $Z$ contains arbitrarily large positive reals, we use
\[
f(z) \sim g(z) \text{\qquad as $z\to\infty$}
\]
to mean that $f(z)/g(z) \to 1$ as $z\to\infty$, and we write $f(z) \asymp g(z)$
to mean that there are positive constants $C_1$ and $C_2$ such that 
$C_1\abs{g(z)} \le \abs{f(z)} \le C_2\abs{g(z)}$ for all sufficiently large~$z$.

When we are working over a finite field $\fq$, we will use without further 
comment the letter $p$ to denote the prime divisor of~$q$. This convention 
unfortunately conflicts with the standard use in analytic number theory of the
letter $p$ as a generic prime, for instance when writing Euler product 
representations of arithmetic functions.  In such situations in this paper
(see for example equation~\eqref{eqhDelta} in Section~\ref{secendomorphisms}),
we will instead use $\ell$ to denote a generic prime, and we explicitly allow 
the possibility that $\ell = p$.

A \emph{curve} over a field $K$ is a smooth, projective, irreducible variety 
over $K$ of dimension one, and a \emph{Jacobian} is the neutral component of 
the Picard scheme of such a curve.

\section{Conjectures of Lang--Trotter type}
\label{seclangtrotter}

Let $\calm$ be a moduli space of PEL type~\cite{Shimura1966}.  Let $K$ be a 
number field, let $\Delta \in \calo_K$ be nonzero, and let $S$ be the set of
primes of $K$ that do not divide $\Delta$. If $\idp$ is a prime of $K$ we let
$\ff_\idp$ denote its residue field.  Equip each finite set $\calm(\ff_\idp)$
with the uniform probability measure, and let $\underline A_\idp$ be a random
variable on $\calm(\ff_\idp)$.  Suppose that for each $\idp \in S$ a subset 
$T_\idp \subset \calm(\ff_\idp)$ is specified, with indicator function 
$I_\idp$.  Let $\underline A = \prod_{\idp \in S} \underline A_\idp$, and let
\[
\pi(\underline A,I,z) = \sum_{\idp \in S \colon \norm(\idp)<z}
                        \frac{\#T_\idp}{\#\calm(\ff_\idp)}
\]
be the expected value of 
$\sum_{\idp \in S \colon  \norm(\idp) \le z}I_\idp(\underline A_\idp)$.
If $\#{T_\idp}/\#{\calm(\ff_\idp)} \asymp 1/\norm(\idp)^m$, then Landau's 
prime ideal theorem~\cite[p.~670]{Landau1903b} yields the estimate 
$\pi(\underline A, I, z) \asymp \int_2^z \frac{dx}{x^m\log x}$. In particular, 
for $m = 1/2$ one finds that $\pi(\underline A,I,z) \asymp \sqrt z/\log z$.   

Henceforth, assume $\calm$ and $T_\idp$ are chosen so that the above holds 
with $m = 1/2$. Now suppose that $A \in \calm(\calo_K[1/\Delta])$, and let
\[
\pi(A,I,z) = \sum_{\idp \in S \colon  \norm(\idp) \le z}I_\idp(A_\idp).
\]
If one \emph{assumes} that ($A$ is sufficiently general, and thus) $A$ is 
well-modeled by the random variable~$\underline A$, then one predicts that
\begin{equation}
\label{eqpredict}
\pi(A,I,z) \asymp \frac{\sqrt z}{\log z}.
\end{equation}
(By ``sufficiently general'' one might mean, for example, that the 
Mumford--Tate group of $A$ is the same as the group attached to the Shimura 
variety $\calm$; but this will not be pursued here.)

For instance, let $\cala_1$ be the moduli stack of elliptic curves, and let 
$a$ be a nonzero integer.  On one hand, since $\cala_1$ is irreducible and 
one-dimensional, we have the estimate $\#\cala_1(\fq)\asymp q$.  On the other 
hand, the number of isomorphism classes of elliptic curves over $\fq$ with 
trace of Frobenius $a$ is the Kronecker class number $H(a^2-4q)$.  Up to 
(at worst) logarithmic factors, the class number $H(a^2-4q)$ grows like 
$\sqrt{\abs{a^2-4q}} \sim 2\sqrt q$ (see~Lemma~\ref{lemclassnumber}).
In this case, the prediction \eqref{eqpredict} yields the Lang--Trotter
conjecture. 

We interpret Theorem~\ref{thmain} as saying that the number of 
principally-polarized split abelian surfaces over $\fq$ is 
approximately~$q^{5/2}$.  This, combined with the fact that $\dim \cala_2 = 3$ 
and thus $\#\cala_3(\fq)\asymp q^3$, is the inspiration behind 
Conjecture~\ref{conjabsurf}. 

In spite of the apparent depth and difficulty of the Lang--Trotter conjecture,
we are certainly not the first to have attempted to formulate analogous 
conjectures in related contexts.  In~\cite{Murty1999}, Murty poses the problem
of counting the primes $\idp$ for which, in a given Galois representation 
$\rho\colon\gal(K) \to \gl_r(\calo_\lambda)$, the trace of Frobenius 
$\tr(\rho(\sigma_\idp))$ is a given number $a$.  The work of Bayer and 
Gonz\'alez~\cite{BayerGonzalez1997} is philosophically more similar to the 
present paper.  Bayer and Gonz\'alez consider a modular abelian variety 
$A/\rat$ and study the number of primes $p$ such that the reduction $A_p$ has
$p$-rank zero.  Unfortunately, in most situations, 
both~\cite[Conj.~8.2, p.~69]{BayerGonzalez1997}
and~\cite[Conj.~2.15, p.~199]{Murty1999} predict a counting function $\pi(z)$
which either grows like $\log\log z$ or is absolutely bounded. In contrast, 
Conjecture~\ref{conjabsurf} has the modest virtue of involving functions that
grow visibly over the range of computationally-feasible values of~$z$.

\section{Split abelian surfaces over finite fields}
\label{secsplit}

In this section we articulate the proof of Theorem~\ref{thmain}, which gives
asymptotic upper and lower bounds on the number of principally-polarized 
abelian surfaces over finite fields such that the abelian surface is isogenous 
to a product of elliptic curves.  There are several different types of such
surfaces, each of which we analyze separately.

First, there are the abelian surfaces over $\fq$ that are isogenous to a 
product $E_1\times E_2$ of two ordinary elliptic curves, with $E_1$ and $E_2$
lying in two different isogeny classes.  We call this the
\emph{ordinary split nonisotypic} case.

\begin{proposition}
\label{proponi}
The number $W_q$ of principally-polarized
ordinary split nonisotypic abelian surfaces
over~$\fq$ satisfies
\[
W_q \ll \begin{cases}
q^{5/2} (\log q)^4 (\log\log q)^2 & \text{for all~$q$, unconditionally,}\\
q^{5/2} (\log q)^2 (\log\log q)^4 & \text{for all~$q$, under GRH.}
\end{cases}
\]
\end{proposition}

Second, there are the abelian surfaces over $\fq$ that are isogenous to the 
square of an ordinary elliptic curve. We call this the
\emph{ordinary split isotypic} case.  

\begin{proposition}
\label{propoi}
The number $X_q$ of principally-polarized ordinary split isotypic abelian 
surfaces over~$\fq$ satisfies
\[
X_q \ll \begin{cases}
q^2 (\log q)^2 \abs{\log\log q}   & \text{for all~$q$, unconditionally,}\\
q^2 (\log q)       (\log\log q)^2 & \text{for all~$q$, under GRH.}
\end{cases}
\]
\end{proposition}

Third, there are the abelian surfaces over $\fq$ that are isogenous to the 
product of two elliptic curves, exactly one of which is supersingular. We call 
this the \emph{almost ordinary split} case.

\begin{proposition}
\label{proppro}
The number $Y_q$ of principally-polarized almost ordinary split abelian 
surfaces over~$\fq$ satisfies
\[
Y_q \ll \begin{cases}
q^2 (\log q)    (\log\log q)^2 & \text{for all $q$, unconditionally,}\\
q^2         \abs{\log\log q}^3 & \text{for all~$q$, under GRH.}
\end{cases}
\]
\end{proposition}

And fourth, there are the abelian surfaces over $\fq$ that are isogenous to 
the product of two supersingular elliptic curves. We call this the
\emph{supersingular split} case.

\begin{proposition}
\label{propss}
The number $Z_q$ of principally-polarized supersingular split abelian surfaces
over $\fq$ satisfies $Z_q\ll q^2$ for all $q$.
\end{proposition}

To prove the lower bound in Theorem~\ref{thmain}, we estimate the number of 
ordinary split nonisotypic surfaces.

\begin{proposition}
\label{proplowerbound}
The number of $W_q$ of principally-polarized ordinary split nonisotypic 
abelian surfaces over $\fq$ satisfies
\[
W_q \gg \begin{cases}
\displaystyle
\frac{q^{5/2}}{(\log q)^3 (\log\log q)^4} & \text{for all~$q$, unconditionally,}\\
\\
\displaystyle
\frac{q^{5/2}}{(\log q)   (\log\log q)^6} & \text{for all~$q$, under GRH.}
\end{cases}
\]
\end{proposition}

It is clear that together these propositions provide a proof of 
Theorem~\ref{thmain}.  We will prove the propositions in the following
sections.  We begin with some background information and results on 
endomorphism rings of elliptic curves over finite fields 
(Section~\ref{secendomorphisms}) and a review of `gluing' elliptic curves 
together (Section~\ref{secgluing}).

\section{Endomorphism rings of elliptic curves over finite fields}
\label{secendomorphisms}

In this section we set notation and give some background information on
endomorphism rings of elliptic curves over finite fields.  With the exception
of the concepts of `strata' and of the `relative conductor', most of the
results on endomorphism rings we mention are standard 
(see~\cite[Ch.~4]{Waterhouse1969} and~\cite{Schoof1987}, and note that
\cite[Thm.~4.5, p.~194]{Schoof1987} corrects a small error
in~\cite[Thm.~4.5, p.~541]{Waterhouse1969}).  

Let $E$ be an elliptic curve over a finite field $\fq$. The substitution 
$x\mapsto x^q$ induces an endomorphism $\frob_{E} \in \End E$ called the 
\emph{Frobenius} endomorphism. The characteristic polynomial of $\frob_{E}$ 
(acting, say, on the $\ell$-adic Tate module of $E$ for some $\ell\neq p$)
is of the form $f_{E}(T) = T^2-a(E)T+q$ for an integer $a(E)$, the 
\emph{trace of Frobenius}. Two elliptic curves $E$ and $E'$ are isogenous if
and only if $a(E) = a(E')$, and Hasse~\cite{Hasse1936a,Hasse1936b,Hasse1936c}
showed that $\abs{a(E)}\le 2 \sqrt q$. We will denote the isogeny class
corresponding to $a$ by
\[
\cali(\fq,a) = \st{E/\fq\colon a(E) = a}.
\]
The isogeny class $\cali(\fq,a)$ is called \emph{ordinary} if $\gcd(a,q) = 1$,
and \emph{supersingular} otherwise 
(see~\cite[p.~526 and Ch.~7]{Waterhouse1969}).  The supersingular curves $E$ 
are characterized by the property that $E[p](\fqbar) \iso \st 0$.

If $E/\fq$ is a supersingular elliptic curve, then $\End_{\fqbar} E$ is a
maximal order in $\rat_{p,\infty}$, the quaternion algebra over $\rat$ ramified
exactly at $\st{p,\infty}$.   There are two possibilities for $\End E$ itself.  
It may be that all of the geometric endomorphisms of $E$ are already defined 
over $\fq$, so that $\End E$ is a maximal order in $\rat_{p,\infty}$; this 
happens when $q$ is a square and $a(E)^2 = 4q$.  The other possibility is that 
$\End E$ is an order in an imaginary quadratic field; in this case, the 
discriminant of $\End E$ is either $-p$, $-4p$, $-3$, or~$-4$.  
(See Table~\ref{tab1} in Section~\ref{secprankone} for the exact conditions
that determine the various cases.)

Suppose $\cali(\fq,a)$ is an isogeny class with $a^2 \neq 4q$.  Then 
$\calo_{a,q} := \integ[T]/(T^2-aT+q)$ is an order in the imaginary quadratic 
field $K_{a,q} := \rat(\sqrt{a^2-4q})$, and is isomorphic to the subring
$\integ[\frob_E]$ of $\End E$ for every $E\in\cali(\fq,a)$.  An order $\calo$ 
in $K_{a,q}$ occurs as $\End E$ for some $E \in \cali(\fq,a)$ if and only if 
$\calo \supseteq \calo_{a,q}$ and $\calo$ is maximal at~$p$ 
(see~\cite[Thm.~4.2, pp.~538--539]{Waterhouse1969} 
or~\cite[Thm.~4.3, p.~193]{Schoof1987}).  Note that the maximality at~$p$ is 
automatic when  $\cali(\fq,a)$ is ordinary, because in that case $q$ is coprime
to the discriminant $a^2 - 4q$ of $\calo_{a,q}$. If we let $\cali(\fq,a,\calo)$
denote the set of isomorphism classes of elliptic curves in $\cali(\fq,a)$ with
endomorphism ring $\calo$, we can write $\cali(\fq,a)$ as a disjoint union
\[
\cali(\fq,a) = \bigsqcup_{\calo\supseteq \calo_{a,q}}
                  \cali(\fq,a,\calo),
\] 
where $\calo$ ranges over all orders of $K_{a,q}$ that contain $\calo_{a,q}$
and that are maximal at~$p$. If $a$ is coprime to~$q$, or if $a=0$ and $q$ is
not a square, then each of the sets $\cali(\fq,a,\calo)$ appearing in the 
equality above is a torsor for the class group $\cl(\calo)$ of the 
order~$\calo$.  In particular, $\#\cali(\fq,a,\calo)$ is equal to the class
number $h(\calo)$ of $\calo$ (see~\cite[Thm.~4.5, p.~194]{Schoof1987}).

We will refer to a nonempty set of the form $\cali(\fq,a,\calo)$ as a
\emph{stratum} of elliptic curves over $\fq$.  Given a stratum $\cals$, we will
denote the associated trace by $a(\cals)$ and the associated quadratic order by
$\calo_\cals$. If $\cals$ and $\cals'$ are two strata over $\fq$, we say that
$\cals$ and $\cals'$ are \emph{isogenous}, and write $\cals\sim\cals'$, if the
elliptic curves in $\cals$ are isogenous to those in $\cals'$ --- that is, if 
$a(\cals) = a(\cals')$.

For any imaginary quadratic order $\calo$ we let $\Delta(\calo)$ denote the
discriminant of $\calo$ and $\Delta^*(\calo)$ the associated fundamental
discriminant --- that is, the discriminant of the integral closure of $\calo$ 
in its field of fractions. Then
\[
\Delta(\calo) = \cond(\calo)^2 \Delta^*(\calo),
\]
where $\cond(\calo)$ is the \emph{conductor} of $\calo$.  For a trace of
Frobenius $a$ with $a^2 \neq 4q$ we will write $\Delta_{a,q}$, $\cond_{a,q}$,
and $\Delta^*_{a,q}$ for the corresponding quantities associated to 
$\calo_{a,q}$. 

Let $E/\fq$ be an elliptic curve whose endomorphism ring is a quadratic order.
We define the \emph{relative conductor} $\relcond(E)$ of $E$ by
\[
\relcond(E) = \frac{\cond(\calo_{a,q})}{\cond(\End E)};
\]
this quantity is also equal to the index of $\calo_{a,q}\cong\integ[\frob_E]$ 
in $\End E$. If $E/\fq$ is a supersingular elliptic curve with endomorphism 
ring equal to an order in a quaternion algebra, we adopt the convention 
$\relcond(E)  = 0$. The relative conductor depends only on the stratum of~$E$,
so for a stratum $\cals$ we may define $\relcond(\cals)$ to be the relative 
conductor of any curve in~$\cals$.

\begin{proposition}
\label{proprelcond}
Let $E/\fq$ be an elliptic curve with $\End E$ a quadratic order.
\begin{alphabetize}
\item 
The relative conductor $\relcond(E)$ is the largest integer $r$ such that 
there exists an integer $b$ with 
\[
\frac{\frob_E - b}{r} \in \End E.
\]
\item 
The relative conductor $\relcond(E)$ is the largest integer $r$ for which 
$\frob_E$ acts as an integer on the group scheme $E[r]$.
\item 
If $E$ is ordinary, the relative conductor $\relcond(E)$ is the largest 
integer $r$, coprime to $q$, for which $\frob_E$ acts as an integer on the 
group $E[r](\fqbar)$.
\end{alphabetize}
\end{proposition}

\begin{proof}
Let $\calo$ be the maximal order containing $\End E$ and let $\omega$ be an 
element of $\calo$ such that $\calo = \integ[\omega]$.  Write 
$\frob_E = u + v\omega$ for integers $u$ and $v$; then 
$\integ[\frob_E] = \integ + v\calo$, so $v = \cond(\integ[\frob_E])$.

On one hand, suppose $r$ is an integer for which there is an integer $b$ with 
$(\frob_E - b)/r \in \End E$.  Then $\End E\supseteq \integ + (v/r)\calo$, 
so $r$ is a divisor of the relative conductor.  On the other hand, if $s$ is
the relative conductor of $E$, then 
$\End E = \integ + (v/s)\calo = \integ[(v/s)\omega]$, 
so $(\frob_E - u) / s$ is an element of $\End E$.  This proves~(a).

If $\frob_E$ acts as an integer $b$ on the group scheme~$E[r]$, then the 
endomorphism $\frob_E - b$ kills $E[r]$.  This implies that $\frob_E - b$ 
factors through multiplication-by-$r$, which means that $(\frob_E - b)/r$ lies 
in $\End E$.  Conversely, if $(\frob_E - b)/r$ lies in $\End E$, then $\frob_E$
acts on $E[r]$ as the integer~$b$.  Thus, (b) follows from~(a).

Suppose $E$ is ordinary. The endomorphism $\frob_E$ does not act as an integer 
on the group scheme $E[p]$, because it acts non-invertibly (consider the local
part of $E[p]$), but not as zero (consider the reduced part of $E[p]$).
Therefore, the integer defined by (b) will not change if we add the requirement
that $r$ be coprime to $p$.  For integers $r$ coprime to $p$, the group scheme 
$E[r]$ is determined by the Galois module $E[r](\fqbar)$. Thus, (c) follows 
from~(b).
\end{proof}

\begin{corollary}
\label{correlcond}
Let $E/\fq$ be an elliptic curve with relative conductor~$r$, and let $n$
be a positive integer.  The largest divisor $d$ of $n$ such that $\frob_E$ 
acts as an integer on $E[d]$ is equal to $\gcd(n,r)$.
\end{corollary}

\begin{proof}
When $\End E$ is a quadratic order, this follows immediately from
Proposition~\ref{proprelcond}.  If the endomorphism ring of $E$ is an order
in a quaternion algebra, then $q$ is a square and $\frob_E = \pm \sqrt{q}$; 
that is, $\frob_E$ \emph{is} an integer, so that $d = n = \gcd(n,0)$.
\end{proof}

Later in the paper we will need to have bounds on the sizes of the
automorphism groups of schemes of the form $E[n]$ for ordinary $E$ and 
positive integers $n$.  Our bounds will involve the Euler function 
$\varphi(n)$ as well as the arithmetic function $\psi$ defined by
$\psi(n) = n \prod_{\ell\div n} (1 + 1/\ell)$.

\begin{proposition}
\label{propautsize}
Let $E$ be an elliptic curve over~$\fq$, let $n$ be a positive integer, and
let $g = \gcd(n,\relcond(E))$. If $E$ is supersingular, assume that $n$ is
coprime to $q$. Then
\begin{equation}
\label{eqnautsize}
\varphi(n) \le \frac{\#\aut E[n]}{g^2 \varphi(n)} \le \psi(n).
\end{equation}
\end{proposition}

\begin{proof}
Every term in the inequality is multiplicative in~$n$, so it suffices to
consider the case where $n$ is a prime power~$\ell^e$.

Suppose $\ell=p$.  In this case, $E$ must be ordinary by assumption.  Note 
that the relative conductor divides the discriminant $a^2 - 4q$, where
$a = a(E)$ is coprime to $p$ because $E$ is ordinary.  Therefore the relative
conductor is coprime to~$p$, so $g = 1$.  

The group scheme $E[n]$ is the product of a reduced-local group scheme $G_1$
and a local-reduced group scheme $G_2$, each of rank~$n$. The group scheme 
$G_1$ is geometrically isomorphic to $\integ/n$, with Frobenius acting as 
multiplication by an integer (which is congruent to $a$ modulo~$q$).  The 
automorphism group of $G_1$ is $(\integ/n)\units$, and has 
cardinality~$\varphi(n)$.

The group scheme $G_2$ is geometrically isomorphic to $\mmu_n$, the group 
scheme of $n$-th roots of unity, with Frobenius acting as power-raising by an 
integer.  The automorphism group of $G_2$ is also $(\integ/n)\units$, and has 
cardinality~$\varphi(n)$.

Since there are no nontrivial morphisms between $G_1$ and $G_2$, the
automorphism group of $E[n]$ is the product of the automorphism groups of $G_1$
and~$G_2$.  Thus, when $n$ is a power of $p$ the middle term 
of~\eqref{eqnautsize} is equal to $\varphi(n)$, and the two inequalities
of~\eqref{eqnautsize} both hold.

Now suppose $\ell\neq p$.  In this case, the group scheme $E[n]$ can be
understood completely in terms of its geometric points and the action of 
Frobenius on them. The group $E[n](\fqbar)$ is isomorphic to $(\integ/n)^2$,
and if we fix such an isomorphism the Frobenius endomorphism is given by an 
element $\gamma$ of $\gl_2(\integ/n)$ whose trace is $a$ and whose determinant
is~$q$.   The automorphism group of $E[n]$ is then isomorphic to the subgroup 
of $\gl_2(\integ/n)$ consisting of those elements that commute with~$\gamma$;
that is, the centralizer $Z(\gamma)$ of $\gamma$.

Let $r$ be the largest divisor of $n$ such that $\frob_E$ acts as an 
integer on $E[r]$; Proposition~\ref{proprelcond} shows that $r = g$. Then 
there is an integer $d$ (uniquely determined modulo~$g$) and a matrix 
$\beta\in \gl_2(\integ/n)$ such that 
$g \cdot \beta \in g \mat_2(\integ/n) \iso\mat_2(\integ/(n/g))$ is cyclic and 
such that $\gamma = d\cdot I + g\cdot\beta$. 
(See \cite{AvniOnnEtAl2009,Williams2012} for details.)

Given this expression for $\gamma$, we can explicitly compute the centralizer
$Z(\gamma)$.  If $g = n$ then $Z(\gamma) = \gl_2(\integ/n)$, so $Z(\gamma)$ has 
order $n\psi(n)\varphi(n)^2$.  If $g$ is a proper divisor of $n$ then 
$Z(\gamma)$ is the group of all $\alpha \in \gl_2(\integ/n)$ such that the
image of $\alpha$ in $\gl_2(\integ/(n/g)) \subset \mat_2(\integ/(n/g))$ lies 
in the $\integ/(n/g)$-span of $I$ and~$\beta$. The order of this subgroup of 
$\gl_2(\integ/(n/g))$ is equal to $\varphi(n/g)$ times
\[
\begin{cases}
   \psi(n/g)  & \text{if $\beta \bmod \ell$ has no eigenvalues in $\integ/\ell$;}\\
        n/g   & \text{if $\beta \bmod \ell$ has 1  eigenvalue  in $\integ/\ell$;}\\
\varphi(n/g)  & \text{if $\beta \bmod \ell$ has 2  eigenvalues in $\integ/\ell$,}
\end{cases}
\]
so the order of its preimage in $\gl_2(\integ/n)$ is either 
$g^2 \psi(n)\varphi(n)$ or $g^2 n\varphi(n)$ or $g^2 \varphi(n)^2$.
In every case we find that 
\[g^2 \varphi(n)^2 \le \#Z(\gamma) \le g^2\psi(n)\varphi(n),\]
which gives~\eqref{eqnautsize}. 
(Alternative methods of calculating $Z(\gamma)$ can be found in~\cite{Williams2012}.)
\end{proof}

Later in the paper we would like to have estimates for the sizes of isogeny
classes and strata; since these sizes are given by class numbers, we close
this section by reviewing some bounds on class numbers.

We denote the class number of an imaginary quadratic order $\calo$ by
$h(\calo)$; this is the size of the group of equivalence classes of invertible
fractional ideals of $\calo$.  We let $H(\calo)$ denote the Kronecker class 
number of~$\calo$, defined by
\[
H(\calo) = \sum_{\calo'\supseteq\calo} h(\calo'),
\]
where the sum is over all quadratic orders that contain $\calo$. If $\Delta$
is the discriminant of an imaginary quadratic order $\calo$, we write 
$h(\Delta)$ and $H(\Delta)$ for $h(\calo)$ and $H(\calo)$, respectively.

\begin{lemma}
\label{lemclassnumber}
We have
\begin{align*}
h(\Delta) &\ll 
  \begin{cases} 
    \abs{\Delta}^{1/2}\log\abs{\Delta} 
       & \text{\quad for fundamental $\Delta<0$},\\
    \abs{\Delta}^{1/2}\log\abs{\Delta} \log\log\abs{\Delta} 
       & \text{\quad for all $\Delta<0$\textup{;}}
\end{cases}\\    
H(\Delta) &\ll \abs{\Delta}^{1/2} \log\abs{\Delta} (\log\log\abs{\Delta})^2
   \text{\qquad for all $\Delta<0$}.
\end{align*}
If the generalized Riemann hypothesis is true, we have
\begin{align*}
h(\Delta) &\ll 
  \begin{cases} 
    \abs{\Delta}^{1/2}\log\log\abs{\Delta} 
       & \text{\quad for fundamental $\Delta<0$},\\
    \abs{\Delta}^{1/2}(\log\log\abs{\Delta})^2
       & \text{\quad for all $\Delta<0$\textup{;}}
\end{cases}\\    
H(\Delta) &\ll \abs{\Delta}^{1/2} (\log\log\abs{\Delta})^3
   \text{\qquad for all $\Delta<0$}.
\end{align*}
\end{lemma}

\begin{proof}
The unconditional bound on $h(\Delta)$ for fundamental $\Delta$ comes 
from~\cite[Exer.~5.27, p.~301]{Cohen1993}, and the conditional bound
from~\cite[Thm.~1, p.~367]{Littlewood1927}.

For an arbitrary negative discriminant $\Delta$, write 
$\Delta = \cond^2\Delta^*$ for a fundamental discriminant $\Delta^*$, and let 
$\chi$ be the quadratic character modulo $\Delta^*$. Then
\begin{equation}
\label{eqhDelta}
h(\Delta) =   \cond h(\Delta^*) \prod_{\ell\div \cond} \left(1 - \frac{\chi(\ell)}{\ell}\right)
          \le \cond h(\Delta^*) \prod_{\ell\div \cond} \left(1 + \oneover{\ell}\right)
          \le       h(\Delta^*) \sigma(\cond),
\end{equation}
where $\sigma$ is the sum-of-divisors function (and we recall that $\ell$ 
ranges over all prime divisors of $\cond$). Since $\sigma(n) \ll n \log\log n$
for $n>2$ by~\cite[Thm.~323, p.~266]{HardyWright1968}, we find that
\[
h(\Delta) \ll \cond h(\Delta^*) \log\log\abs{\Delta} \text{\qquad for all $\Delta<0$.}
\]
Combining this with the class number bounds for fundamental discriminants
gives us the bounds for arbitrary discriminant.  

For Kronecker class numbers, note that
\begin{align*}
H(\Delta) = \sum_{f\div\cond} h(f^2\Delta^*)
          = \sum_{f \div \cond} f h(\Delta^*) \prod_{\ell\div f} \left(1 - \frac{\chi(\ell)}{\ell}\right)
          \le h(\Delta^*) \bigg(\sum_{f \div \cond}f\bigg)\prod_{\ell\div \cond} \left(1 + \oneover{\ell}\right)
          \le \cond\inv h(\Delta^*) \sigma(\cond)^2,
\end{align*}
so that 
\[
H(\Delta) \ll \cond h(\Delta^*) (\log\log\abs{\Delta})^2\text{\qquad for all $\Delta<0$.}
\]
This leads to the desired bounds on $H(\Delta)$.
\end{proof}

\section{Gluing elliptic curves}
\label{secgluing}

In this section, we review work of Frey and Kani~\cite{FreyKani1991} that 
explains how to construct principally-polarized abelian surfaces from pairs 
of elliptic curves provided with some extra structure. First, we discuss 
isomorphisms of torsion subgroups of elliptic curves.

Let $E$ and $F$ be elliptic curves over a field $K$ and let $n>0$ be an 
integer. We let $\isom(E[n],F[n])$ denote the set of group scheme isomorphisms
between the $n$-torsion subschemes of $E$ and~$F$. The Weil pairing gives us 
nondegenerate alternating pairings
\[
E[n]\times E[n] \to  \mmu_n
\text{\quad and\quad} 
F[n]\times F[n] \to  \mmu_n
\]
from the $n$-torsion subschemes of $E$ and of $F$ to the $n$-torsion of the 
multiplicative group scheme. Via the Weil pairing, we get a map 
\[ 
m\colon \isom(E[n],F[n]) \to \aut \mmu_n \iso (\integ/n\integ)\units.
\]
For every $i\in (\integ/n\integ)\units$ we let $\isom^i(E[n],F[n])$ denote
the set $m\inv(i)$, so that $\isom^1(E[n],F[n])$ consists of the group scheme 
isomorphisms that respect the Weil pairing, and $\isom\inv(E[n],F[n])$
consists of the anti-isometries from $E[n]$ to $F[n]$.

If $\eta$ is an anti-isometry from $E[n]$ to $F[n]$, then the graph $G$ of 
$\eta$ is a subgroup scheme of $(E\times F)[n]$ that is maximal isotropic 
with respect to the product of the Weil pairings.  It follows 
from~\cite[Cor.~to Thm.~2, p.~231]{Mumford1974} that $n$ times the canonical 
principal polarization on $E\times F$ descends to a principal polarization 
$\lambda$ on the abelian surface $A := (E\times F)/G$.  In this situation, 
we say that the polarized surface $(A,\lambda)$ is obtained by \emph{gluing} 
$E$ and $F$ together along their $n$-torsion subgroups via~$\eta$.

Frey and Kani~\cite{FreyKani1991} show that every principally-polarized
abelian surface $(A,\lambda)$ that is isogenous to a product of two elliptic
curves arises in this way; furthermore, if such an $A$ is not isogenous to 
the square of an elliptic curve, then the $E$, $F$, $n$, and $\eta$ that give
rise to the polarized surface $(A,\lambda)$ are unique up to isomorphism and 
up to interchanging the triple $(E, F, \eta)$ with $(F, E, \eta\inv)$.

Frey and Kani also note that if the polarized surface $(A,\lambda)$
constructed in this way is the canonically-polarized Jacobian of a curve~$C$,
then there are minimal degree-$n$ maps $\alpha\colon C\to E$ and 
$\beta\colon C\to F$ such that $\alpha_*\beta^* = 0$;  here \emph{minimal} 
means that $\alpha$ and $\beta$ do not factor through nontrivial isogenies. 
Conversely, every pair of minimal degree-$n$ maps $\alpha\colon C\to E$ 
and $\beta\colon C\to F$ such that $\alpha_*\beta^* = 0$ arises in this way.

\section{Ordinary split nonisotypic surfaces}
\label{seconi}

In this section we will prove Proposition~\ref{proponi}.  The proof depends on
three lemmas, whose proofs we postpone until the end of the section.

\begin{lemma}
\label{lemboundnonisotypic}
The number $W_q$ of principally-polarized ordinary split nonisotypic
abelian surfaces over $\fq$ is at most
\[
\sum_{\cals} \sum_{\cals'\nsim \cals} 
     h(\calo_\cals)h(\calo_{\cals'})
     \relcond(\cals)\relcond(\cals') 
     \sum_{n\div(a(\cals)-a(\cals'))} \psi(n),
\]
where the first sum is over ordinary strata $\cals$, and the second is over
ordinary strata $\cals'$ not isogenous to $\cals$.
\end{lemma}

\begin{lemma}
\label{lemboundsumpsi}
We have
\[
\sum_{d\div n} \psi(d) \ll  n (\log \log n)^2 \text{\qquad for all $n>1$.}
\]
\end{lemma}

\begin{lemma}
\label{lemboundsumrelcond}
We have
\[
\sum_{\textup{ordinary }E/\fq} \relcond(E) \ll 
\begin{cases}
q (\log q)^2                  & \text{for all $q$, unconditionally,}\\
q (\log q)  \abs{\log\log q}  & \text{for all $q$, under GRH.}
\end{cases}
\]
\end{lemma}

Given these lemmas, the proof of Proposition~\ref{proponi} is straightforward.

\begin{proof}[Proof of Proposition~\textup{\ref{proponi}}]
From Lemmas~\ref{lemboundnonisotypic} and~\ref{lemboundsumpsi} we find  that
\[
W_q \ll q^{1/2} (\log \log q)^2 
        \sum_{\cals} \sum_{\cals'\nsim \cals} 
        h(\calo_\cals)h(\calo_{\cals'}) \relcond(\cals)\relcond(\cals') 
        \text{\qquad for all $q$}.
\]
Since 
\[
\sum_{\cals} \sum_{\cals'\nsim \cals} 
        h(\calo_\cals)h(\calo_{\cals'}) \relcond(\cals)\relcond(\cals') 
 < \bigg( \sum_{\cals} h(\calo_\cals) \relcond(\cals) \bigg)^2 
 = \bigg( \sum_{\textup{ordinary }E/\fq} \relcond(E) \bigg)^2,
\]
we have
\[
W_q  \ll q^{1/2} (\log \log q)^2 
        \bigg( \sum_{\textup{ordinary } E/\fq} \relcond(E)  \bigg)^2
        \text{\qquad for all $q$.}
\]
Combining this with Lemma~\ref{lemboundsumrelcond}, we find that we have
\[
W_q \ll
\begin{cases}
q^{5/2} (\log q)^4 (\log \log q)^2 & \text{for all $q$, unconditionally,}\\
q^{5/2} (\log q)^2 (\log \log q)^4 & \text{for all $q$, under GRH.}
\end{cases}\qedhere
\]
\end{proof}

Now we turn to Lemmas~\ref{lemboundnonisotypic}, \ref{lemboundsumpsi}, 
and~\ref{lemboundsumrelcond}.  The proof of Lemma~\ref{lemboundnonisotypic}
itself requires some notation and a preparatory result.

Fix an elliptic curve $E/\fq$ and a stratum $\cals$ of elliptic curves
over~$\fq$.  For a positive integer $n$, let 
\begin{align*}
\isom(E,\cals,n)     &= \st{(E,E',\eta): E' \in \cals, \eta\in \isom(E[n],E'[n])} \\
\isom\inv(E,\cals,n) &= \st{(E,E',\eta): E' \in \cals, \eta\in \isom\inv(E[n],E'[n])}.
\end{align*}

\begin{lemma}
\label{lemboundantiisometries}
Suppose that either $\cals$ is ordinary, or that $a(\cals) = 0$ and $q$ is a
nonsquare. If $\isom\inv(E,\cals,n)$ is nonempty then
$\gcd(n,\relcond(E)) = \gcd(n,\relcond(\cals))$, and we have
\[
\# \isom\inv(E,\cals,n) \le 2 \psi(n) h(\calo_\cals) 
               \gcd(n,\relcond(E)) \gcd(n, \relcond(\cals)).
\]
In particular, if $\relcond(E)\neq 0$, then
\[
\# \isom\inv(E,\cals,n) \le 2 \psi(n) h(\calo_\cals) \relcond(E) \relcond(\cals).
\]
\end{lemma}

\begin{proof}
Suppose that $\isom\inv(E,\cals,n)$ is nonempty.  Then there is an $E'\in\cals$
for which there is an isomorphism $E[n]\cong E'[n]$. Corollary~\ref{correlcond}
then shows that 
$\gcd(n,\relcond(E)) = \gcd(n,\relcond(E')) = \gcd(n,\relcond(\cals))$.

The class group $\cl(\calo_\cals)$ acts on $\cals$, and the assumption that 
either $\cals$ is ordinary or that $a(\cals) = 0$ and $q$ is a nonsquare 
implies that $\cals$ is a torsor for the class group. Define an action of 
$\aut E[n] \times \cl(\calo_\cals)$ on the nonempty set $\isom(E,\cals,n)$ by
setting
\[
(\alpha, [\ida])\comp (E,E',\eta) 
  = (E, [\ida]* E', [\ida]\comp \eta \comp \alpha\inv).
\]
It is clear that $\isom(E,\cals,n)$ is a torsor for 
$\aut E[n] \times \cl(\calo_\cals)$ under this action, so using 
Proposition~\ref{propautsize} we find that
\[
\# \isom(E,\cals,n) \le \big(\#\aut E[n]\big) \, h(\calo_\cals)
                    \le g^2 \varphi(n)\psi(n) h(\calo_\cals),
\]
where $g = \gcd(n,\relcond(E)).$ Therefore
\[
\#\isom(E,\cals,n)
  \le \varphi(n)\psi(n) h(\calo_\cals) \gcd(n,\relcond(E)) \gcd(n,\relcond(\cals)).
\]

In the preceding section we defined a map 
$m\colon\isom(E[n],E'[n])\to \aut\mmu_n$ that sends a group scheme isomorphism 
to the automorphism of $\mmu_n$ induced by the Weil pairing.  This gives rise 
to a map from $\isom(E,\cals,n)$ to $\aut\mmu_n$, which we continue to denote
by~$m$, that sends a triple $(E,E',\eta)$ to~$m(\eta)$. We claim that the image
of this map is a coset of a subgroup of $\aut \mmu_n$ of index at most~$2$.

To see this, we use the theory of complex multiplication, the
Galois-equivariance of the Weil pairing, and class field theory for the
extension $\rat(\zeta_n)/\rat$ as follows. Let $K$ be the field of fractions
of $\calo_\cals$.  Given $[\ida] \in \cl(\calo_\cals)$ and 
$(E,E',\eta) \in \isom(E,\cals,n)$, we have
\[
m( (1, [\ida])\comp (E,E',\eta)) =
(\norm_{K/\rat}(\ida),\rat(\zeta_n)/\rat) \comp m(\eta) \in \aut \mmu_n,
\]
where $(\ \cdot\ ,\rat(\zeta_n)/\rat)$ denotes the Artin symbol for the 
extension $\rat(\zeta_n)/\rat$. Since the group of norms of id\`ele classes of 
$K$ has index $[K:\rat] = 2$ in the group of id\`ele classes of~$\rat$, the 
image of the map $m$ is a coset of a subgroup of index at most~$2$.

Therefore, the number of elements in $\isom\inv(E,\cals,n)$ is at most 
$2/\varphi(n)$ times the number of elements in $\isom(E,\cals,n)$, and we 
obtain the inequality in the lemma.
\end{proof}

\begin{proof}[Proof of Lemma~\textup{\ref{lemboundnonisotypic}}]
As we noted in Section~\ref{secgluing}, every principally-polarized ordinary 
split nonisotypic surface over $\fq$ is obtained in exactly two ways by gluing 
two ordinary nonisogenous curves $E$ and $E'$ together along their $n$-torsion.
Since we must then have $E[n]\cong E'[n]$, the traces of Frobenius of $E$ and
$E'$ must be congruent to one another modulo $n$; that is, 
$n \div (a(E) - a(E'))$. Summing over ordinary $E$ and $E'$, we find that
\begin{alignat*}{2}
2W_q &=     \sum_{E}     \sum_{E'    \nsim E} \quad \sum_{n\div (a(E) - a(E'))    } \#\isom\inv(E[n],E'[n])\\
     &=     \sum_{E}     \sum_{\cals'\nsim E} \quad \sum_{n\div (a(E) - a(\cals'))} \#\isom\inv(E,\cals',n)\\
     &\le   \sum_{E}     \sum_{\cals'\nsim E} \quad \sum_{n\div (a(E) - a(\cals'))} 2\psi(n) h(\calo_{\cals'})\relcond(E)\relcond(\cals') & \text{\qquad (by Lemma~\ref{lemboundantiisometries})}\\
     &\le 2 \sum_{E}     \sum_{\cals'\nsim E}                   h(\calo_{\cals'})\relcond(E)\relcond(\cals') \sum_{n\div (a(E) - a(\cals'))} \psi(n)\\
     &=   2 \sum_{\cals} \sum_{\cals'\nsim \cals} h(\calo_\cals)h(\calo_{\cals'})\relcond(\cals)\relcond(\cals') \sum_{n\div (a(\cals) - a(\cals'))} \psi(n),\\
\end{alignat*}
which proves the lemma.
\end{proof}

\begin{proof}[Proof of Lemma~\textup{\ref{lemboundsumpsi}}]
Denote the sum on the left by $f(n)$, so that $f$ is a multiplicative function.
We calculate that 
$f(n)/n \le \prod_{\ell\div n} (1+\frac{1}{\ell})/(1-\frac{1}{\ell}).$
Taking this inequality and multiplying by the square of the identity 
$\varphi(n)/n = \prod_{\ell\div n} (1-\frac{1}{\ell})$, we find that
\[
\frac{f(n)}{n(\log\log n)^2}\left(\frac{\varphi(n)\log\log n}{n}\right)^2 \le
\prod_{\ell\div n}\left(1 - \frac{1}{\ell^2}\right) \le 1.
\]
Landau~\cite{Landau1903a} showed that 
$\liminf \varphi(n)(\log\log n)/n = e^{-\gamma}$, where $\gamma$ is Euler's 
constant.  The lemma follows.
\end{proof}

Our proof of Lemma~\ref{lemboundsumrelcond} requires an estimate from
analytic number theory.  Let $C$ be the multiplicative arithmetic function
defined on prime powers $\ell^e$ by $C(\ell^e) = 2(1 + 1/\ell)$.

\begin{lemma}
\label{lemCbound}
We have
\[
\sum_{n\le x} C(n) \ll x \log x \text{\qquad for all $x > 1$.}
\]
\end{lemma}

\begin{proof}
Let $D$ be the Dirichlet product (\cite[\S2.6]{Apostol1976}) of $C$ with the 
M\"obius function $\mu$, so that
\[
C(n) = \sum_{d\div n} D(d).
\]
We compute that $D$ is the multiplicative function defined on prime powers 
$\ell^e$ by
\[
D(\ell^e) = \begin{cases}
            1 + 2/\ell    & \text{if $e = 1$,}\\
            0          & \text{if $e > 1$.}\\
            \end{cases}
\]
Then
\[
\sum_{n\le x}C(n) 
     =   \sum_{n\le x} \sum_{d\div n} D(d)
     =   \sum_{d\le x} D(d) \bigfloor{\frac{x}{d}}
     \le x \sum_{d\le x}\frac{D(d)}{d},
\]
so we need only show that $\sum_{d\le x} D(d)/d \ll \log x $ for $x > 1$.

Note that 
\[
\sum_{i = 0}^\infty \frac{D(\ell^i)}{\ell^i} = 1 + \oneover{\ell} + \frac{2}{\ell^2},
\]  
so that
\[\sum_{d\le x}\frac{D(d)}{d} \le 
  \prod_{\ell \le x} \left( 1 + \oneover{\ell} + \frac{2}{\ell^2}\right).
\]
Taking logarithms, we find that 
\begin{align*}
\log\sum_{d\le x}\frac{D(d)}{d}
          &\le \sum_{\ell\le x} \log\left(1 + \oneover{\ell} + \frac{2}{\ell^2}\right) \\
          &=\sum_{\ell\le x} \oneover{\ell} + c + O\left(\oneover{x}\right)\\
          &=\log\log x + c' + O\left(\oneover{\log x}\right),
\end{align*}
where $c$ and $c'$ are constants and where the last equality comes 
from~\cite[Thm.~4.12, p.~90]{Apostol1976}.  Exponentiating, we find that
$\sum_{d\le x} D(d)/d \ll \log x$ for $x\ge 2$, as desired.
\end{proof}

\begin{proof}[Proof of Lemma~\textup{\ref{lemboundsumrelcond}}]
First we compute a bound on the sum of the relative conductors of the elliptic
curves in a fixed ordinary isogeny class.  Let $a$ be an integer, coprime 
to~$q$, with $a^2 < 4q$.  Recall from Section~\ref{secendomorphisms} that we 
write $\Delta_{a,q} := a^2 - 4q = \cond_{a,q}^2 \Delta_{a,q}^*$, where 
$\Delta_{a,q}^*$ is a fundamental discriminant.  Let $\tilde\calo_{a,q}$
be the quadratic order of discriminant $\Delta_{a,q}^*$.  As we noted in 
Section~\ref{secendomorphisms}, the isogeny class $\cali(\fq,a)$ is the
union of strata $\cals = \cali(\fq,a,\calo)$, where the orders 
$\calo\subseteq\tilde\calo_{a,q}$ have discriminant $f^2 \Delta_{a,q}^*$ for the 
divisors $f$ of $\cond_{a,q}$.  The curves in $\cals$ have relative conductor 
$\cond_{a,q}/f$, and the number of curves in $\cals$ is equal to $h(\calo)$.
If we let $\chi$ denote the quadratic character modulo $\Delta_{a,q}^*$, then
\[
h(\calo) = f h(\Delta^*_{a,q}) \prod_{\ell\div f} \left(1 - \frac{\chi(\ell)}{\ell}\right).
\]
Thus,
\[
\sum_{E\in\cali(\fq,a)}\relcond(E)
= 
\sum_{f\div \cond_{a,q}} \frac{\cond_{a,q}}{f} f h(\Delta^*_{a,q}) 
   \prod_{\ell\div f} \left(1 - \frac{\chi(\ell)}{\ell}\right)
= \cond_{a,q}  h(\Delta^*_{a,q}) \sum_{f\div \cond_{a,q}}
   \prod_{\ell\div f} \left(1 - \frac{\chi(\ell)}{\ell}\right).
\]
Lemma~\ref{lemclassnumber} tells us that 
$h(\Delta) \ll \abs{\Delta}^{1/2} \log\abs{\Delta}$ for all fundamental
discriminants $\Delta<0$.  Combining this with the fact that 
$\abs{f_{a,q}^2\Delta_{a,q}^*} = 4q - a^2 < 4q$ we see that there is a
constant $c$ such that for all $q$ and $a$, we have
\[
\sum_{E\in\cali(\fq,a)}\relcond(E) < c q^{1/2} (\log q) A(\cond_{a,q}),
\]
where $A$ is the arithmetic function defined by
\[
A(n) = \sum_{d\div n} \prod_{\ell\div d} \left(1 + \frac{1}{\ell}\right)
     = \sum_{d\div n} \frac{\psi(d)}{d}.
\]
Additionally, if the generalized Riemann hypothesis is true we can use 
Lemma~\ref{lemclassnumber} to find that there is a constant $c'$ such that for
all $q$ and $a$ we have
\[
\sum_{E\in\cali(\fq,a)}\relcond(E) < c' q^{1/2} \abs{\log \log q} A(\cond_{a,q}).
\]
Thus, to prove the lemma it will suffice to show that we have
\begin{equation}
\label{eqsuffice}
\sum_{\substack{1\le a \le {2\sqrt{q}} \\ \gcd(a,q)=1}}
      A(\cond_{a,q}) \ll q^{1/2} \log q
      \text{\qquad for all $q$.}
\end{equation}

Note that the sum on the left side of~\eqref{eqsuffice} is equal to
\[
\sum_{\substack{1\le a \le {2\sqrt{q}} \\ \gcd(a,q)=1}}
     \quad \sum_{d\div \cond_{a,q}} \frac{\psi(d)}{d} = 
\sum_{1\le d\le 2\sqrt{q}} \frac{\psi(d)}{d} 
     \# \st{a \colon 1\le a \le 2\sqrt{q} \text{ and } \gcd(a,q)=1
                                          \text{ and } d\div \cond_{a,q} }.
\]
If $d\div \cond_{a,q}$ then $a^2 \equiv 4q \bmod d^2$, so let us first 
consider, for a fixed $d$, estimates for the number of $a$ in the interval 
$[1, 2\sqrt{q}]$ with $a^2 \equiv 4q \bmod d^2$.

We have
\begin{align*}
\# \st{a \colon 1\le a \le 2\sqrt{q} \text{ and } a^2 \equiv 4q \bmod d^2 }
 &\le \# \st{a \colon 1\le a \le d^2 \ceil{2\sqrt{q} / d^2} \text{ and } a^2 \equiv 4q \bmod d^2} \\
 &= \ceil{2\sqrt{q} / d^2} \# \st{a \colon 1\le a \le d^2 \text{ and } a^2 \equiv 4q \bmod d^2}.
\end{align*}
Thus, if we let $B_q$ denote the multiplicative arithmetic function given by
\[
B_q(n) = \# \st{ a \colon 1\le a \le n^2 \text{ and } a^2 \equiv 4q \bmod n^2 }
\]
then we have
\begin{align}
\notag \sum_{\substack{1\le a \le 2\sqrt{q} \\ \gcd(a,q)=1}}  A(\cond_{a,q}) 
                &\le \sum_{\substack{d \le 2\sqrt{q} \\ \gcd(d,q)=1}} 
                           \frac{\psi(d)}{d} 
                           \# \st{ a \colon 1\le a \le 2\sqrt{q} \text{ and } \gcd(a,q)=1 \text{ and } d\div \cond_{a,q} }\\
\notag          &\le \sum_{\substack{d \le 2\sqrt{q} \\ \gcd(d,q)=1}}
                           \frac{\psi(d)}{d} \ceil{2\sqrt{q} / d^2} B_q(d)\\
\label{eqABsum} 
                &\le \sum_{\substack{d \le 2\sqrt{q} \\ \gcd(d,q)=1}}
                           \frac{ 2\sqrt{q}} {d^2} \frac{\psi(d)}{d} B_q(d) + 
                     \sum_{\substack{d \le 2\sqrt{q} \\ \gcd(d,q)=1}}
                           \frac{\psi(d)}{d} B_q(d).
\end{align}

If $\ell$ is a prime that does not divide $q$ and if $e>0$ then
\[
B_q(\ell^e) \le \begin{cases}
              2 & \text{if $\ell\ne 2$}\\
              8 & \text{if $\ell = 2$},
            \end{cases}
\]
so 
\[
\frac{\psi(d)}{d} B_q(d) \le 4 C(d)
\]
for all $d$ coprime to~$q$, where $C$ is the function from
Lemma~\ref{lemCbound}. For every $\epsilon>0$ we have $C(d) \ll d^\epsilon$
for all~$d$, so
\begin{equation}
\label{eqABbound1}
\sum_{\substack{d \le 2\sqrt{q} \\ \gcd(d,q)=1}}
                           \oneover{d^2} \frac{\psi(d)}{d} B_q(d) 
\le 4 \sum_{\substack{d \le 2\sqrt{q} \\ \gcd(d,q)=1}}
                           \frac{C(d)}{d^2}
\le 4 \sum_{d=1}^\infty \frac{C(d)}{d^2}                 
< \infty   \text{\qquad for all $q$}.                       
\end{equation}
This shows that the first term on the right side of~\eqref{eqABsum} is 
$\ll \sqrt{q}$ for all~$q$.  To bound the second term on the right side 
of~\eqref{eqABsum}, we compute that 
\begin{equation}
\label{eqABbound2}
\sum_{\substack{d \le 2\sqrt{q} \\ \gcd(d,q)=1}}  \frac{\psi(d)}{d} B_q(d)
 \le 4 \sum_{d\le 2\sqrt{q} } C(d) \ll q^{1/2} \log q
\text{\qquad for all $q$},
\end{equation}
by Lemma~\ref{lemCbound}. Combining~\eqref{eqABsum} with~\eqref{eqABbound1}
and~\eqref{eqABbound2} proves~\eqref{eqsuffice}, and completes the proof of 
the lemma.
\end{proof}

\section{Ordinary split isotypic surfaces}
\label{secoi}

In this section we will prove Proposition~\ref{propoi}. As in the 
preceding section, we state several lemmas which lead to a quick 
proof of the proposition. Lemma~\ref{lemsumisogenousrelcond} follows 
from Lemma~\ref{lemboundantiisometries}.  We postpone the proofs of
Lemmas~\ref{lemmapdegree}, \ref{lemboundpsisum},  
and~\ref{lemsmallestisogenies} until the end of the section.

\begin{lemma}
\label{lemsumisogenousrelcond}
For every ordinary $E / \fq$ and positive integer $n$ we have
\[
\pushQED{\qed}
\sum_{E'\sim E} \#\isom\inv(E,E',n)  \le 2 \psi(n) \relcond(E)
\sum_{E'\sim E} \relcond(E').\qedhere
\popQED
\]
\end{lemma}

\begin{lemma}
\label{lemmapdegree}
Let $E/\fq$ be an elliptic curve and let $C/\fq$ be a smooth genus-$2$ curve
with $\jac C \sim E^2$. Then there is a finite morphism $C \to E$ of degree at
most~$\sqrt{2q}$. If $E$ is supersingular with all endomorphisms defined 
over~$\fq$, then there is a finite morphism $C \to E$ of degree at 
most~$q^{1/4}$.
\end{lemma}

\begin{lemma}
\label{lemboundpsisum}
We have
\[
\sum_{n \le x} \psi(n) = \frac{15}{2\pi^2} x^2 + O(x \log x).
\]
\end{lemma}

For every pair of isogenous curves $E$ and $E'$ over $\fq$, we let $s(E,E')$
denote the degree of the smallest isogeny from $E$ to $E'$.

\begin{lemma}
\label{lemsmallestisogenies}
Let  $E/\fq$ be an ordinary elliptic curve with $\relcond(E)= 1$, and let 
$\cals$ be a stratum of curves isogenous to~$E$.  Then
\[ 
\sum_{E' \in \cals} \frac{1}{s(E,E')^2} < \frac{\zeta(3)}{\relcond(\cals)^2},
\]
where $\zeta$ is the Riemann zeta function.
\end{lemma}

\begin{proof}[Proof of Proposition~\textup{\ref{propoi}}]
Proposition~\ref{propoi} gives an upper bound on the number of
principally-polarized abelian surfaces isogenous to the square of an ordinary
elliptic curve.  We would like to instead consider Jacobians. This requires 
that we first dispose of those principally-polarized surfaces that are not 
Jacobians of smooth curves; according to 
\cite[Thm.~3.1, p.~270]{GonzalezGuardiaEtAl2005}, these are the polarized 
surfaces that are products of elliptic curves with the product polarization,
together with the restrictions of scalars of polarized elliptic curves over the
quadratic extension of our base field. But the restriction of scalars of an 
elliptic curve over $\ff_{q^2}$ with trace of Frobenius $b$ is an abelian
surface over $\fq$ with Weil polynomial $x^4 - bx^2 + q^2$, and such a surface
is never isogenous to the square of an ordinary elliptic curve, because in that
case its Weil polynomial would have to be $(x^2 - ax + q)^2$ where $a$ is
coprime to~$q$.  Therefore, to dispose of the non-Jacobians, we need only 
consider products of elliptic curves, with the product polarization.

The number of elliptic curves in an ordinary isogeny class with trace of 
Frobenius equal to $a$ is equal to the Kronecker class number $H(a^2-4q)$ of
the discriminant $a^2 - 4q$ (see~\cite[Thm.~4.6, pp.~194--195]{Schoof1987}).
From Lemma~\ref{lemclassnumber} we know that
$
H(\Delta) \ll \abs{\Delta}^{1/2} \log\abs{\Delta} (\log\log\abs{\Delta})^2
$
for all negative discriminants $\Delta$.  Therefore the number of product 
surfaces $E\times E'$ with $E$ and $E'$ both in a fixed ordinary isogeny class
over $\fq$ is $\ll q (\log q)^2 (\log \log q)^4$; summing over isogeny
classes, we find that the number of product surfaces $E\times E'$ with $E$ and
$E'$ ordinary and isogenous to one another is 
$\ll q^{3/2} (\log q)^2 (\log \log q)^4.$  Thus, the contribution of the 
non-Jacobians to the ordinary split isotypic polarized surfaces is much less 
than the bound claimed in Proposition~\ref{propoi}. (Of course, for present
purposes, it suffices to observe that the number of non-Jacobians is bounded 
by the square of the number of elliptic curves over $\fq$; but the estimate 
provided here is closer to the actual truth.)

Fix an integer $a$ with $\abs a \le 2\sqrt q$ and $\gcd(a,q) =1$, and let $E_a$
be an elliptic curve over $\fq$ with $a(E) = a$ and with 
$\End E_a \iso \calo_{a,q}$, so that $\relcond(E_a) = 1$. Suppose $C$ is a
curve over $\fq$ whose Jacobian is isogenous to $E_a^2$. By 
Lemma~\ref{lemmapdegree} there is a morphism $\phi$ from $C\to E_a$ of degree 
at most $\sqrt{2q}$.  We can write this map as a composition of a minimal map
$C\to E$ (see Section~\ref{secgluing}) with an isogeny $E\to E_a$, and it 
follows that the degree of the minimal map $C\to E$ is at most
$\sqrt{2q} / s(E,E_a)$. If we let $N_a$ denote the number of genus-$2$ curves
with Jacobians isogenous to $E_a^2$, we find that
{\allowdisplaybreaks
\begin{align}
\notag
N_a &\le \sum_{E\sim E_a} 
     \#\st{\text{$C$ with minimal maps to $E$ of degree at most $\sqrt{2q}/s(E,E_a)$}}\\
\notag
    &\le \sum_{E\sim E_a} \quad \sum_{n\le\sqrt{2q}/s(E,E_a)} \quad
          \sum_{E'\sim E} \#\isom\inv(E,E',n)\\
\label{eqline3}          
    &\le \sum_{E\sim E_a} \quad \sum_{n\le\sqrt{2q}/s(E,E_a)} \quad
          2 \psi(n) \relcond(E) \sum_{E'\sim E} \relcond(E')\\
\notag
    &= 2\bigg(\sum_{E'\sim E_a} \relcond(E')\bigg)
        \sum_{E\sim E_a} \relcond(E) \sum_{n\le\sqrt{2q}/s(E,E_a)} \psi(n)\\
\label{eqline5}
    &\ll \bigg(\sum_{E'\sim E_a} \relcond(E')\bigg)
        \sum_{E\sim E_a} \relcond(E) \frac{2q}{s(E,E_a)^2}
         \text{\qquad for all $a$ and $q$}.
\end{align}
}%
Here~\eqref{eqline3} follows from Lemma~\ref{lemsumisogenousrelcond}
and~\eqref{eqline5} follows from Lemma~\ref{lemboundpsisum}. Now we group the 
curves $E$ isogenous to $E_a$ by their strata. Recall that we have 
$a^2 - 4q = \cond_{a,q}^2 \Delta^*_{a,q}$, and that the strata of curves 
isogenous to $E_a$ are indexed by the divisors $f$ of $\cond_{a,q}$.  We find 
that
\begin{alignat}{2}
\notag
N_a & \ll q \bigg(\sum_{E'\sim E_a} \relcond(E')\bigg)
       \sum_{\cals\sim E_a} \relcond(\cals) \sum_{E\in\cals} \frac{1}{s(E,E_a)^2}
     &\qquad\text{for all $a$ and $q$}\\
\label{eqline2}
    & \ll q \bigg(\sum_{E'\sim E_a} \relcond(E')\bigg)
       \sum_{f\div \cond_{a,q}} \frac{1}{f}
     &\text{for all $a$ and $q$}\\
\label{eqline3b}
    & \ll q \bigg(\sum_{E'\sim E_a} \relcond(E')\bigg) \abs{\log \log q},
     &\text{for all $a$ and $q$}&,
\end{alignat}       
where \eqref{eqline2} follows from Lemma~\ref{lemsmallestisogenies}
and~\eqref{eqline3b} follows from the asymptotic upper bound 
\cite[Thm.~323, p.~266]{HardyWright1968} 
\[
e^\gamma = \limsup_{n>0} \frac{\sum_{d\div n} d}{n \log\log n}
         = \limsup_{n>0} \frac{\sum_{d\div n} d/n}{\log\log n}
         = \limsup_{n>0} \frac{\sum_{d\div n} 1/d}{\log\log n}.
\]

Recall that $X_q$ is the number of principally-polarized ordinary split isotypic
abelian surfaces over $\fq$. Then $X_q$ is the sum over all $a$ coprime to $q$ 
of the $N_a$ (together with the negligible contribution from those abelian 
surfaces that are isomorphic, as principally-polarized abelian varieties, to 
products of isogenous elliptic curves), and we find that
\[
X_q \ll q \abs{\log\log q}\bigg(\sum_{\textup{ordinary }E/\fq} \relcond(E)\bigg)
\text{\qquad for all $q$}.
\]
Proposition~\ref{propoi} then follows from
Lemma~\ref{lemboundsumrelcond}.
\end{proof}

\begin{proof}[Proof of Lemma~\textup{\ref{lemmapdegree}}]
Choose a divisor of degree $1$ on $C$, and let $L$ be the additive group of
morphisms from $C$ to $E$ that send the given divisor to the identity of~$E$.
Let $\bolde$ be the base extension of $E$ from $\fq$ to the function field $F$ 
of~$C$.  The \emph{Mordell--Weil lattice} of $\bolde$ over $F$ is the group
$\bolde(F)/E(\fq)$ provided with the pairing coming from the canonical height.
The natural map $L\to \bolde(F)/E(\fq)$ is a bijection, and the quadratic form 
on $L$ obtained from the height pairing on $\bolde(F)$ is twice the degree map
(see~\cite[Thm.~III.4.3, pp.~217--218]{Silverman1994}). Let $a = a(E)$, and let
$\pi$ and $\pibar$ be the roots in $\complexes$ of the characteristic polynomial
of Frobenius for $E$, so that $\pi+\pibar = a$. The Birch and Swinnerton-Dyer
conjecture for constant elliptic curves over function fields (proved by 
Milne~\cite[Thm.~3, pp.~100--101]{Milne1968}) shows that the determinant of the
Mordell--Weil lattice is a divisor of 
\[
\begin{cases}
(\pi-\pibar)^4 = (a^2 - 4q)^2  & \text{if $\pi\neq\pibar$,}\\
q^2 & \text{if $\pi=\pibar$};
\end{cases}
\]
note that $\pi=\pibar$ if and only if $E$ is supersingular with all of its
endomorphisms rational over $\fq$.

The $\integ$-rank of $L$ is twice the $\integ$-rank of $\End E$. If 
$\pi\neq\pibar$, so that $\End E$ is an imaginary quadratic order, then $L$ is
a $\integ$-module of rank~$4$. Applying~\cite[Thm.~12.2.1, p.~260]{Cassels1978}
we find that there is a nonzero element of $L$ of degree at most
\[
\frac{1}{2} \gamma_4 \abs{a^2-4q}^{1/2},
\]
where $\gamma_4$ is the Hermite constant for dimension~$4$. Using the fact that
$\gamma_4 = \sqrt{2}$ (see~\cite{Hermite1850}), we obtain the bound in the 
lemma.

If $\pi=\pibar$ then $\End E$ is an order in a quaternion algebra and $L$ is a
$\integ$-module of rank~$8$. We find that there is a nonzero element of $L$ with
degree at most
\[
\frac{1}{2} \gamma_8 q^{1/4}.
\]
The value of $\gamma_8$ was determined by Blichfeldt~\cite{Blichfeldt1935}
to be~$2$, so there is a map from $C$ to $E$ of degree at most~$q^{1/4}$.
\end{proof}

\begin{proof}[Proof of Lemma~\textup{\ref{lemboundpsisum}}]
First we note that
\[
\psi(n) = \sum_{d\div n} \abs{\mu(d)} \frac{n}{d},
\]
where $\mu$ is the M\"obius function.  Then, arguing as in the proof 
of~\cite[Thm.~3.7, p.~62]{Apostol1976}, we see that
\begin{align*}
\sum_{n\le x}\psi(n) & = \sum_{\substack{d,q\\ dq\le x}} \abs{\mu(d)} q 
                       = \sum_{d\le x}\abs{\mu(d)} \sum_{q\le x/d} q \\
  &= \sum_{d\le x} \abs{\mu(d)}\left(\frac{1}{2}\left(\frac{x}{d}\right)^2 
                                   + O\left(\frac{x}{d}\right)\right)\\
  &= \frac{1}{2} x^2 \sum_{d\le x} \frac{\abs{\mu(d)}}{d^2} 
                                 + O\left(x\sum_{d\le x}\frac{1}{d}\right)\\
  &= \frac{1}{2} c x^2 + O(x\log x),\\
\end{align*}
where 
\[
c = \sum_{d=1}^\infty \frac{\abs{\mu(d)}}{d^2} 
  = \prod_{\ell} \left(1 + \frac{1}{\ell^2}\right)
  = \prod_{\ell} \frac{\left(1 - {1}/{\ell^4}\right)}{\left(1 - {1}/{\ell^2}\right)}
  = \frac{\zeta(2)}{\zeta(4)} = \frac{15}{\pi^2}.\qedhere
\]
\end{proof}
  
\begin{proof}[Proof of Lemma~\textup{\ref{lemsmallestisogenies}}]
Let $\calo = \calo_\cals$ be the order corresponding to the stratum $\cals$.
We claim that there is an elliptic curve $\tilde{E}$ in $\cals$ and an isogeny
$f\colon E\to \tilde{E}$ with the property that every isogeny from $E$ to an 
elliptic curve in $\cals$ factors through $f$.  One way to see this is via the
theory of Deligne modules~\cite{Deligne1969,Howe1995}.  If we let $\pi$ be the 
Frobenius for $E$ and let $K$ be the quadratic field $\rat(\pi)$, then the
Deligne modules of the elements of $\cals$ can be viewed as lattices in $K$
with endomorphism rings equal to $\calo$, while the Deligne module for $E$ can
be viewed as a lattice $\Lambda\subset K$ with $\End \Lambda =\integ[\pi]$.  
The curve $\tilde{E}$ is the elliptic curve corresponding to the Deligne module
$\Lambda\otimes\calo$, and the isogeny $f$ corresponds to the inclusion 
$\Lambda\subset\Lambda\otimes\calo$. In particular, we see that the degree of
$f$ is equal to~$\relcond(\cals)$.

The isogenies from $\tilde{E}$ to the other elements of $\cals$ correspond to
the invertible ideals $\ida\subset\calo$, with different ideals giving rise to 
different isogenies.  Let $\ida_1,\ldots,\ida_n$ be the (distinct) ideals 
corresponding to the smallest isogenies from $\tilde{E}$ to the elements of 
$\cals$, where $n = \#\cals$. Then
\[
\sum_{E'\in \cals} \frac{1}{s(E,E')^2}
 = \sum_{i=1}^n \frac{1}{\relcond(\cals)^2 \norm(\ida_i)^2}
 < \frac{1}{\relcond(\cals)^2} \sum_{\text{all $\ida$}} \frac{1}{\norm(\ida)^2},
\]
where the final sum is over all invertible ideals $\ida\subseteq\calo$; that 
is, the final sum is equal to $\zeta_\calo(2)$, where $\zeta_\calo$ is the zeta
function for the order $\calo$.

Kaneko~\cite[Proposition, p.~202]{Kaneko1990} gives an explicit formula for 
$\zeta_\calo(s)$ in terms of the zeta function for $K$ and the conductor 
of~$\calo$.  It is not hard to check that for real $s>1$, the Euler factor at
$\ell$ for $\zeta_\calo(s)$ is bounded above by $1/(1-\ell^{1-2s})$, so that 
$\zeta_\calo(s) < \zeta(2s-1)$, where $\zeta$ is the Riemann zeta function.  
In particular, $\zeta_\calo(2) < \zeta(3)$, and the lemma follows.
\end{proof}

\section{Almost ordinary split surfaces}
\label{secprankone}

In this section we prove Proposition~\ref{proppro}, which gives an upper bound
on the number of principally-polarized almost ordinary split abelian surfaces.
We base the proof on two lemmas, which we prove at the end of the section.

\begin{lemma}
\label{lemfrobint}
Let $E_0$ be a supersingular elliptic curve over a finite field~$\fq$, with $q$
a square, and suppose the Frobenius endomorphism on $E_0$ is equal to
multiplication by~$s$, where $s^2 = q$. Let $\cals$ be a stratum of ordinary
elliptic curves over $\fq$, and suppose $n$ is a positive integer such that 
$\isom\inv(E_0,\cals,n)$ is nonempty.  Then
\begin{alphabetize}
\item the integer $n$ is coprime to $q$,
\item the relative conductor $\relcond(\cals)$ is divisible by $n$, and
\item the trace $a(\cals)$ satisfies $4a(\cals) \equiv 8s \bmod n^2$.
\end{alphabetize}
\end{lemma}

\begin{lemma}
\label{lemboundpsisum2}
We have
\[
\sum_{n \le x} \frac{\psi(n)}{n}= \frac{15}{\pi^2} x + O(\log x).
\]
\end{lemma}

\begin{proof}[Proof of Proposition~\textup{\ref{proppro}}] 
In analogy with Section~\ref{seconi}, we will bound the number $Y_q$ of 
principally-polarized almost ordinary split abelian surfaces over $\fq$ by 
estimating the number of surfaces obtained by gluing a supersingular $E_0$ to 
an ordinary $E$ along their $n$-torsion subgroups.  The methods we use will 
depend on whether or not $E_0$ has all of its endomorphisms defined over~$\fq$.
Let $Y_{q,1}$ denote the number of principally-polarized surfaces we get from
$E_0$ with all of the endomorphisms defined, and let $Y_{q,2}$ denote the 
number we get from $E_0$ with not all endomorphisms defined.  We will show that 
$Y_{q,1}$ and $Y_{q,2}$ each satisfy the bound of Proposition~\ref{proppro}.

First let us bound $Y_{q,1}$; that is, we consider the case where all of the
endomorphisms of $E_0$ are defined over~$\fq$. In this case, $q$ is a square 
and the characteristic polynomial of $E_0$ is $(T - 2s)^2$, where $s^2 = q$;
furthermore,~\cite[Thm.~4.6, pp.~194--195]{Schoof1987} tells us that there are 
\[
\oneover{12}\left(p + 6 - 4\quadres{-3}{p} - 3\quadres{-4}{p} \right) 
  \le\frac{\sqrt{q}}{2}
\]
such curves for each of the two possible values of $s$, so at most $\sqrt{q}$
curves in total.

Fix such an $E_0$ and fix an integer $n>0$.  Suppose $\cals$ is an ordinary 
stratum of elliptic curves over $\fq$ such that $\isom\inv(E_0,\cals,n)$ is
nonempty.  If $n$ is even let $m = n/2$; otherwise let $m = n$.  We see from 
Lemma~\ref{lemfrobint} that the trace $a(\cals)$ of $\cals$ is an integer
congruent to $2s$ modulo $m^2$, but not equal to~$2s$. The number of such 
integers $a$ in the Weil interval is at most $\floor{4\sqrt{q}/m^2}$.

Given such an integer $a$, write $a^2 - 4q = \cond_{a,q}^2 \Delta_{a,q}^*$
for a fundamental discriminant $\Delta_{a,q}^*$.  Let $\chi$ be the quadratic
character modulo $\Delta_{a,q}^*$, and for each divisor $d$ of $\cond_{a,q}/n$
let $\cals_d$ be the stratum $\cals$ with $a(\cals) = a$ and $\cond(\cals) = d$.
Using Lemma~\ref{lemboundantiisometries} we find that
\begin{align*}
\sum_{E\in\cali(\fq,a)} \#\isom\inv(E_0[n], E[n])
  & = \sum_{\substack{\cals\text{ with}\\ a(\cals) = a \text{ and } n\div\relcond(\cals)}}
       \#\isom\inv(E_0,\cals,n)\\
  & = \sum_{ d \div (\cond_{a,q}/n)} \#\isom\inv(E_0,\cals_d,n)\\
  & \le 2 \sum_{ d \div (\cond_{a,q}/n)}  \psi(n) h(\calo_{\cals_d}) n^2\\
  & = 2 \psi(n) n^2 H\left(\frac{a^2 - 4q}{n^2}\right),\\
\end{align*}           
where $H(x)$ is the Kronecker class number.  Thus,
\[
\sum_{E\in\cali(\fq,a)} \#\isom\inv(E_0[n], E[n])
\ll\begin{cases}
2 \psi(n) n q^{1/2}(\log q)    (\log \log q)^2 & \text{for all $a$ and $q$, unconditionally,}\\
2 \psi(n) n q^{1/2}        \abs{\log \log q}^3 & \text{for all $a$ and $q$, under GRH.}
\end{cases}
\]
Summing over the  $\floor{4\sqrt{q}/m^2}$ possible values of $a$ for a
given~$n$, and then summing over the possible $n < 4\sqrt{q}$, and then summing
over the possible curves~$E_0$, we find that 
\[
Y_{q,1} \ll q^{3/2} (\log q)(\log \log q)^2 \sum_{n=1}^{4\sqrt{q}} \frac{\psi(n)}{n}
        \ll q^{2} (\log q)(\log \log q)^2 
        \text{\qquad for all $q$}.
\]
If the generalized Riemann hypothesis holds, we get the better bound
\[
Y_{q,1} \ll q^{2} \abs{\log \log q}^3 \text{\qquad for all $q$}.
\]

Now we turn to estimating $Y_{q,2}$, the number of principally-polarized split
surfaces isogenous to a surface of the form $E_0\times E$, where $E$ is 
ordinary and $E_0$ is supersingular with not all endomorphisms defined.
Using~\cite[Thms.~4.2, 4.3, and~4.5, pp.~194--195]{Schoof1987}, we find that
the possible strata of such curves $E_0$ are as listed in Table~\ref{tab1}.

\smallskip
\begin{table}[ht]
\begin{center}
\renewcommand{\arraystretch}{1.4}
\begin{tabular}{lllcccc}
\toprule
Conditions on $q$ & Conditions on $p$  &&     $a(\cals)$ & $\Delta(\calo_\cals)$ & $\relcond(\cals)$ & $\#\cals$           \\
\midrule
$q$ nonsquare     & ---                &&      $0$       &         $-4p$         &   $ \sqrt{q/p}$   & $h(-4p)$            \\
                  & $p\equiv 3\bmod 4$ &&      $0$       &          $-p$         &   $2\sqrt{q/p}$   & $h( -p)$            \\
                  & $p = 2$            && $\pm\sqrt{2q}$ &          $-4$         &   $ \sqrt{q/2}$   & $1$                 \\
                  & $p = 3$            && $\pm\sqrt{3q}$ &          $-3$         &   $ \sqrt{q/3}$   & $1$                 \\
$q$ square        & ---                &&      $0$       &          $-4$         &   $ \sqrt{q  }$   & $1 - \quadres{-4}{p}$ \\
                  & ---                && $\pm\sqrt{ q}$ &          $-3$         &   $ \sqrt{q  }$   & $1 - \quadres{-3}{p}$ \\
\bottomrule
\end{tabular}
\vskip 3ex
\end{center}
\caption{The supersingular strata $\cals$ over $\fq$ with not all endomorphisms
         defined over $\fq$.  Here $q$ is a power of a prime $p$.}
\label{tab1}         
\end{table}

Let $E_0$ be a supersingular curve with not all endomorphisms defined. If we 
are to glue $E_0$ to an ordinary elliptic curve $E$ along the $n$-torsion of 
the two curves, then $n$ must be coprime to~$q$. In that case, the greatest 
common divisor of $\relcond(E_0)$ and $n$ is either $1$ or $2$, as we see from
Table~\ref{tab1}.  It follows from Lemma~\ref{lemboundantiisometries} that for 
every ordinary stratum $\cals$, we have
\[
\#\isom\inv(E_0,\cals,n) \le 8 \psi(n) h(\calo_\cals),
\]
so the total number of curves obtained from gluing $E_0$ to an ordinary
elliptic curve is bounded by 
\[
8 \sum_{\text{ordinary }\cals}  h(\calo_\cals) \sum_{n\div (a(\cals) - a(E_0))} \psi(n)
\ll
q^{1/2} (\log\log q)^2 \sum_{\text{ordinary }\cals}  h(\calo_\cals)
\text{\qquad for all $q$}
\]
by Lemma~\ref{lemboundsumpsi}.  This last sum is simply the number of ordinary
elliptic curves over $\fq$, which (one shows) is at most $2q + 4$, so the 
number of curves obtained as above from a fixed $E_0$ is 
$\ll q^{3/2} (\log\log q)^2$ for all $q$.

If $q$ is a square there are at most $6$ possible $E_0$, and we find that
\[
Y_{q,2} \ll q^{3/2} (\log\log q)^2\text{\qquad for all square $q$}.
\]
If $q$ is not a square, then Lemma~\ref{lemclassnumber}
shows that the number of possible $E_0$ is $\ll q^{1/2} \log q$ for all $q$
unconditionally, and $\ll  q^{1/2} \abs{\log \log q}$ for all $q$ under the
generalized Riemann hypothesis. This leads to
\[
Y_{q,2}\ll\begin{cases}
q^2 (\log q)    (\log\log q)^2   & \text{for all $q$ unconditionally,}\\
q^2         \abs{\log\log q}^3   & \text{for all $q$ under GRH,}
\end{cases}
\]
and completes the proof of Proposition~\ref{proppro}.
\end{proof}

\begin{proof}[Proof of Lemma~\textup{\ref{lemfrobint}}]
Since $\isom\inv(E_0,\cals,n)$ is nonempty, there is an $E \in \cals$ with 
$E_0[n]\cong E[n]$.  The $p$-torsion of $E_0$ is a local-local group scheme, 
while $E[p]$ has no local-local part, so $n$ must not be divisible by $p$. 
This proves~(a).

Lemma~\ref{lemboundantiisometries} shows that 
$\gcd(n,\relcond(E_0)) = \gcd(n,\relcond(\cals))$.  Since $\relcond(E_0) = 0$,
we find that $n\div \relcond(\cals)$.  This proves~(b).

Let $a = a(\cals)$.  From (b) we know that $a^2 - 4q \equiv 0 \bmod n^2$, and
we also know that $a \equiv a(E_0) = 2s \bmod n$. Since $a - 2s \equiv 0\bmod n$
we have 
\[
0 \equiv a^2 - 4as + 4s^2
  \equiv 4s^2 - 4as + 4s^2 
  \equiv 8s^2 - 4as \bmod n^2.
\]
Since $s$ is coprime to $n$ by (a), we can divide through by $s$ to obtain~(c).
\end{proof}

\begin{proof}[Proof of Lemma~\textup{\ref{lemboundpsisum2}}]
The proof is quite similar to that of Lemma~\ref{lemboundpsisum}. We have
\[
\sum_{n\le x}\frac{\psi(n)}{n} 
  = \sum_{\substack{d,q\\ dq\le x}} \frac{\abs{\mu(d)}}{d} 
  = \sum_{d\le x}\frac{\abs{\mu(d)}}{d} \left\lfloor \frac{x}{d}\right\rfloor 
  = x \sum_{d\le x} \frac{\abs{\mu(d)}}{d^2}  + O(\log x)
  = c x + O(\log x)
\]
where 
$c = \sum_{d=1}^\infty \abs{\mu(d)}/d^2 = 15/\pi^2.$
\end{proof}

\section{Supersingular split surfaces}

In this section we prove Proposition~\ref{propss}, which gives a bound on the 
number $Z_q$ of principally-polarized supersingular split abelian surfaces 
over~$\fq$. 

We must first introduce some terminology and some background results. Let $A$
be an abelian surface over a finite field $\fq$ of characteristic~$p$, and let
$\aalpha_p$ denote the (unique) local-local group scheme of rank $p$ over~$\fq$.
The \emph{$a$-number} of $A$ is the dimension of the $\fq$-vector space
$\Hom(\aalpha_p,A)$. If $A$ has $a$-number $2$ then $A$ is called 
\emph{superspecial}; all superspecial surfaces over $\fq$ are geometrically 
isomorphic to one another, and they are all geometrically isomorphic to the 
square of a supersingular elliptic curve.  A supersingular surface $A$ has 
$a$-number equal to either $1$ or $2$; if the $a$-number is~$1$, then $A$ has a
unique local-local subgroup scheme of rank~$p$, and the quotient of $A$ by this
subgroup scheme is a superspecial surface.

Let $\cala_2^\supersing$ denote the supersingular locus of the coarse moduli
space of principally-polarized abelian surfaces.  
Koblitz~\cite[p.~193]{Koblitz1975} shows that the only singularities of 
$\cala_2^\supersing$ are at the superspecial points, and 
from~\cite[Proof of Cor.~4.7, p.~117]{Oort1974} we know that each irreducible
component of $\cala_{2,\fpbar}^\supersing$ is a curve of genus~$0$. Also, every 
component contains a superspecial point.  Therefore, the non-superspecial locus
of $\cala_{2,\fpbar}^\supersing$ is a disjoint union of components, each of
which is isomorphic to an open affine subset of~$\affine^1$.

Moreover, the number of irreducible components of $\cala_{2,\fpbar}^\supersing$
is equal to the class number $H_2(1,p)$ of the non-principal genus of 
$\rat^2_{p,\infty}$ (see~\cite[Thm.~5.7, p.~133]{KatsuraOort1987}). Hashimoto
and Ibukiyama~\cite{HashimotoIbukiyama1981} 
(see also~\cite[Rmk.~2.17, p.~147]{IbukiyamaKatsuraEtAl1986}) provide a formula
for $H_2(1,p)$ which shows both that $H_2(1,p) = p^2/2880 + O(p)$ and that 
$H_2(1,p) \le p^2/4$ for all~$p$.

For convenience, we also state the following lemma.
\begin{lemma}
\label{lemforms}
Let $(A,\lambda)$ be a principally-polarized abelian surface over $\fqbar$ that
has a model over $\fq$.  Then the number of distinct $\fq$-rational models of 
$(A,\lambda)$ is at most $1152$.
\end{lemma}

\begin{proof}
The size of the automorphism group of a principally-polarized abelian surface 
over a finite field is bounded by $1152$ (by $72$, if the characteristic is 
greater than $5$); for Jacobians, this follows from Igusa's enumeration of the
possible automorphism groups~\cite[\S8]{Igusa1960}, and for products of 
polarized elliptic curves and for restrictions of scalars of elliptic curves it 
is an easy exercise. (We know from 
\cite[Thm.~3.1, p.~270]{GonzalezGuardiaEtAl2005} that every 
principally-polarized abelian surface is of one of these three types.)
By~\cite[Lemma~7.2, pp.~85--86]{BrockGranville2001}, the number of 
$\fq$-rational forms of such a polarized surface is bounded by this same number.
\end{proof}

With these preliminaries out of the way, we may proceed to the proof of 
Proposition~\ref{propss}.  The proof splits into cases, depending on whether or
not the base field is a prime field.  First we consider the case where $q$
ranges over the set of primes~$p$.

We may assume that $p>3$.  In that case, we see from Table~\ref{tab1} that 
there is only one isogeny class of supersingular elliptic curves, the isogeny
class $\cali(\ff_p,0)$ of trace-$0$ curves, which consists of $1$ or $2$ strata.

Pick a trace-$0$ elliptic curve $E_0/\ff_p$ whose endomorphism ring has
discriminant $-4p$.  If $(A,\lambda)$ is a principally-polarized abelian 
surface over $\ff_p$ with $A$ isogenous to $E_0^2$, then either $A$ is a
product of elliptic curves with the product polarization, or $A$ is the 
restriction of scalars of an elliptic curve over $\ff_{p^2}$ with trace $-2p$,
or $A$ is the Jacobian of a curve~$C$.  
(See \cite[Thm.~3.1, p.~270]{GonzalezGuardiaEtAl2005}.)
The number of elliptic curves in $\cali(\ff_p,0)$ is $H(-4p)$; using 
Lemma~\ref{lemclassnumber} we see that the number of products of such elliptic 
curves is $\ll p (\log p)^2 (\log \log p)^4 \ll p^2$ for all primes~$p$. The 
number of supersingular elliptic curves over $\ff_{p^2}$ with trace $-2p$ is
equal to the number of supersingular $j$-invariants, which is is $p/12 + O(1)$;
therefore the number of restrictions of scalars of such curves is $\ll p$. 
Thus, we may focus our attention on the case where $(A,\lambda)$ is the 
Jacobian of a curve~$C$.

In this case, we know from Lemma~\ref{lemmapdegree} that $C$ has a map of 
degree at most $\sqrt{2p}<p$ to $E_0$, so $C$ has a \emph{minimal} map of 
degree at most $p$ to a curve $E$ in $\cali(\ff_p,0)$. We see that the 
polarized variety $(A,\lambda)$ can be obtained by gluing together two elliptic 
curves $E$ and $E'$ in $\cali(\ff_p,0)$ along their $n$-torsion, for some
$n < p$. It follows that the $a$-number of $A$ is~$2$, so $A$ is superspecial.
By \cite[Rem.~3, p.~41]{IbukiyamaKatsura1994}, the number of 
principally-polarized superspecial abelian surfaces over $\fpbar$ which admit
a model over $\ff_p$ is $\ll  p h(-p),$ which in turn is 
$\ll p^{3/2}(\log p) \abs{\log\log p}$ by Lemma~\ref{lemclassnumber}. By 
Lemma~\ref{lemforms}, we get the same bound for the number of superspecial
curves over~$\ff_p$. This shows that Proposition~\ref{propss} holds as $q$ 
ranges over the set of primes.

Now we let $q$ range over the set of proper prime powers.  Let $q = p^e$ for 
some prime $p$ and $e>1$. First we bound the number of principally-polarized
superspecial split surfaces.

By \cite[Thm.~2, p.~41]{IbukiyamaKatsura1994}, the total number of superspecial
curves over $\fqbar$ is equal to the class number $H_2(1,p)\le p^2/4$ mentioned
above, so by Lemma~\ref{lemforms} there are at most $1152 p^2/4 = 288 p^2$
superspecial curves over~$\fq$. Similarly, the number of supersingular
$j$-invariants is $p/12+O(1)$, so the number of distinct products of polarized 
supersingular elliptic curves over $\fqbar$ is also bounded by a constant
times~$p^2$; by Lemma~\ref{lemforms}, this shows that the number of 
principally-polarized superspecial split abelian surfaces over $\fq$ that are
not Jacobians is $\ll p^2$. Since $q\ge p^2$, the number of 
principally-polarized superspecial split surfaces is~$\ll q$.

We are left with the task of estimating the number of non-superspecial
supersingular split curves over $\fq$.  To do this, we appeal to a moduli space 
argument.  As noted above, the coarse moduli space of non-superspecial 
supersingular curves is geometrically a union of  $p^2/2880 + O(p)$ components,
each one an open subvariety of $\affine^1$.  Thus, the number of $\fq$-rational
points on this moduli space is at most $p^2 q/2800 + O(pq)$. By 
Lemma~\ref{lemforms}, each rational point on the moduli space corresponds
to at most $1152$ curves over $\fq$, so there are $\ll p^2 q \ll q^2$ 
principally-polarized supersingular split abelian surfaces over~$\fq$.  
\qed

\section{A lower bound for the number of split surfaces}
\label{seclowerbound}

In this section we prove Proposition~\ref{proplowerbound}.

Let $\ell$ be a prime coprime to $q$.  We say that two elliptic curves $E$ and 
$F$ over $\fq$ are \emph{of the same symplectic type modulo $\ell$} if (in the 
notation of Section~\ref{secgluing}) the set $\isom^{1}(E[\ell],F[\ell])$ is
nonempty; that is, if there is an isomorphism $E[\ell]\to F[\ell]$ of group 
schemes that respects the Weil pairing.  Clearly, if $E$ and $F$ have the same 
symplectic type modulo $\ell$ then their traces of Frobenius are congruent 
modulo $\ell$, so for each residue class modulo~$\ell$, the elliptic curves
whose traces lie in that residue class are distributed among some number of 
symplectic types.

\begin{lemma}
\label{lemtypes}
Let $\ell$ be an odd prime coprime to $q$ and let $a\in\integ/\ell$.
\begin{alphabetize}
\item If $a^2 \not\equiv 4q \bmod \ell$ then all elliptic curves $E/\fq$
      with $a(E) \equiv a\bmod \ell$ are of the same symplectic type.
\item If $a^2 \equiv 4q \bmod \ell$, there are at most three
      symplectic types of elliptic curves with trace congruent to $a$.
      If we fix an $\ell$th root of unity $\zeta\in\fqbar$, 
      these three types are determined as follows\textup{:}
      \begin{itemize}
      \item[1.] Those $E$ for which Frobenius acts as an integer on $E[\ell]$.
      \item[2.] Those $E$ for which Frobenius does not act as an integer
            on $E[\ell]$, and for which the Weil pairing $e(P,\frob_E(P))$
            is of the form $\zeta^x$ with $x\in(\integ/\ell)\units$ a square
            for all $P\in E[\ell](\fqbar)$ with $\frob_E(P)\ne (a/2)P$.
      \item[3.] Those $E$ for which Frobenius does not act as an integer
            on $E[\ell]$, and for which the Weil pairing $e(P,\frob_E(P))$
            is of the form $\zeta^x$ with $x\in(\integ/\ell)\units$ a nonsquare
            for all $P\in E[\ell](\fqbar)$ with $\frob_E(P)\ne (a/2)P$.
      \end{itemize}
\end{alphabetize}      
\end{lemma}

\begin{corollary}
\label{cortypes}
\pushQED{\qed}
For each odd $\ell$ coprime to $q$,
there are at most $\ell+4$ symplectic types of
elliptic curves modulo $\ell$ over~$\fq$.\qedhere
\popQED
\end{corollary}

\begin{proof}[Proof of Lemma~\textup{\ref{lemtypes}}]
Let $E$ be an elliptic curve over $\fq$ and let $G$ be the automorphism group
of $E[\ell]$. In Section~\ref{secgluing} we defined a map 
$m\colon G\to\aut\mmu_\ell$.  If $a^2 \not\equiv 4q \bmod \ell$ then $m$ is 
surjective, so there is an isometry between $E[\ell]$ and $F[\ell]$ for any 
two curves $E$ and $F$ of trace $a$. Likewise, if Frobenius acts as a constant
on $E[\ell]$ then $m$ is surjective, so if Frobenius acts as $a/2$ on $E[\ell]$
and $F[\ell]$ then there is an isometry between those two group schemes.  

On the other hand, if Frobenius does not act semisimply then the image of $m$
is a coset of a subgroup of index $2$, that is, a coset of the subgroup of 
squares, and is an isometry between $E[\ell]$ and $F[\ell]$ for two such curves
$E$ and $F$ if and only if the image of $m$ is the same for both of them.
\end{proof}

\begin{lemma}
\label{lemminustypes}
Let $\ell$ be a prime coprime to $q$ and with $\ell\equiv 1\bmod 4$.  If two
elliptic curves $E$ and $F$ over $\fq$ have the same symplectic type 
modulo~$\ell$, then there are at least $\ell-1$ elements of 
$\isom\inv(E[\ell], F[\ell]).$
\end{lemma}

\begin{proof}
Since $E$ and $F$ have the same symplectic type modulo $\ell$ there is an 
isometry $\eta\colon E[\ell]\to F[\ell]$. Let $b$ be an integer with 
$b^2 \equiv -1\bmod\ell$. Then $b\eta$ is an anti-isometry, so
$\isom\inv(E[\ell],F[\ell])$ is nonempty. From Proposition~\ref{propautsize}
we know that $\#\aut E[\ell] \ge (\ell-1)^2$, so $\#\isom^{1}(E[\ell],E[\ell])$
is at least $\ell-1$, and it follows that there are at least this many elements
of $\isom\inv(E[\ell],F[\ell])$.
\end{proof}

\begin{proof}[Proof of Proposition~\textup{\ref{proplowerbound}}]
Let $c$ be a constant such that 
\[
H(\Delta) < c \abs{\Delta}^{1/2} \log\abs{\Delta} (\log\log\abs{\Delta})^2
\]
for all negative discriminants $\Delta$; such a constant exists by 
Lemma~\ref{lemclassnumber}.  We will show that for every prime $\ell \ne p$ 
with $\ell\equiv1\bmod 4$ and with
\begin{equation}
\label{eqellbound}
\ell < \frac{q^{1/2}}{1600 c^2 (\log q)^2(\log\log q)^4}
\end{equation}
there are more than $2 q^2/5 $ triples $(E_1,E_2,\eta)$, where $E_1$ and $E_2$ 
are nonisogenous ordinary elliptic curves over $\fq$ and 
$\eta\colon E_1[\ell]\to E_2[\ell]$ is an anti-isometry.  Dirichlet's theorem 
shows that there are constants $c'\ge 13, c'' > 0$ such that when $q\ge c'$ the 
number of such primes $\ell$ is at least
\[
\frac{c'' q^{1/2}}{(\log q)^3(\log\log q)^4},
\]
so for $q\ge c'$ we will have at least
\[
\frac{c'' q^{5/2}}{5 (\log q)^3 (\log\log q)^4}
\]
distinct principally-polarized abelian surfaces, thus proving the unconditional
part of Proposition~\ref{proplowerbound}.

Let $\ell$ be a prime as above, let $t\le \ell+4$ be the number of symplectic
types of curves modulo~$\ell$, and let $S_1,\ldots,S_t$ be the sets of ordinary 
curves of the $t$ different symplectic types. We would like to count the number 
of pairs of curves $(E_1, E_2)$ where $E_1$ and $E_2$ are not isogenous to one 
another but are of the same symplectic type. The number of ordered pairs 
$(E_1,E_2)$ where $E_1$ and $E_2$ are of the same type is 
$\sum_{i=1}^t (\#S_i)^2.$  This sum is minimized when the elliptic curves are 
evenly distributed across the symplectic types.  It is easy check that when 
$q\ge 13$ there are always at least $5q/3$ ordinary elliptic curves over~$\fq$, 
so we see that
\[
\sum_{i=1}^t (\#S_i)^2 \ge t \left(\frac{5q}{3t}\right)^2 
\ge \frac{25q^2}{9(\ell+4)}
\ge  \frac{125 q^2}{81 \ell} > \frac{3 q^2}{2 \ell}.
\]

On the other hand, the number of ordered pairs $(E_1,E_2)$ of ordinary elliptic
curves that are isogenous to one another is
\[
\sum_{\substack{-2\sqrt{q}< a < 2\sqrt{q}\\ \gcd(a,q)=1}} H(a^2-4q)^2.
\]
Using Lemma~\ref{lemclassnumber} and the definition of $c$, we see that each
summand is at most
\[
 c^2 (4q)(\log (4q))^2 (\log \log (4q))^4
 < 400 c^2  q (\log q)^2 (\log \log q)^4.
 \]
so the number of such ordered pairs is at most
\[
1600 c^2 q^{3/2} (\log q)^2 (\log \log q)^4 \le q^2 / \ell.
\]
Thus, the number of ordered pairs $(E_1,E_2)$ of nonisogenous curves that have
the same symplectic type is at least $(1/2)(q^2/\ell)$. By 
Lemma~\ref{lemminustypes}, this gives us more than  
$(1/2)(\ell-1) q^2/ \ell > 2q^2/5$ triples $(E_1,E_2,\eta)$ where $E_1$ and
$E_2$ are nonisogenous ordinary elliptic curves and 
$\eta\colon E_1[\ell]\to E_2[\ell]$ is an anti-isometry, as we wanted.

If the generalized Riemann hypothesis holds, we modify our argument as 
follows.   We take $c$ to be a constant such that
\[
H(\Delta) < c \abs{\Delta}^{1/2}(\log\log\abs{\Delta})^3
\]
for all negative discriminants $\Delta$, and consider primes 
$\ell\equiv 1\bmod 4$ bounded by
\[
\ell < \frac{q^{1/2}}{1600 c^2 (\log\log q)^6}
\]
instead of by~\eqref{eqellbound}.  Again we find that for each such $\ell$ we
have more than $2 q^2/5$ triples $(E_1,E_2,\eta)$, where $E_1$ and $E_2$ are 
nonisogenous ordinary elliptic curves and $\eta\colon E_1[\ell]\to E_2[\ell]$
is an anti-isometry. Dirichlet's theorem then leads to the desired estimate
for~$W_q$.
\end{proof}

\section{Numerical data, evidence for Conjecture~\ref{conjabsurf}, and further
directions}
\label{secdata}

In this section we present summaries of some computations that help give some 
indication of the behavior of several of the quantities that we study and
provide bounds for, and we give some evidence that seems to support 
Conjecture~\ref{conjabsurf}. We close with some thoughts about possible 
extensions of our results.

\subsection{The sum of the relative conductors}
\label{ssec-data-relcond}

In Section~\ref{seconi} we proved Proposition~\ref{proponi}, which gives an 
upper bound on the number of principally-polarized ordinary split nonisotypic
abelian surfaces over a finite field $\fq$. The key to the argument is 
Lemma~\ref{lemboundsumrelcond}, which gives an upper bound for the sum of the
relative conductors of the ordinary elliptic curves over~$\fq$. The lemma shows
that there is a constant $c$ such that for all $q$ this sum is at most 
$c q (\log q)^2$.  However, we suspect that the sum of the relative conductors
grows more slowly than this; it is perhaps even~$O(q)$. 

We computed this sum for all prime powers $q$ less than~$10^7$.
For $q$ in the range $(10^3,10^4)$, the sum lies between $2.07 q$ and $4.27 q$; 
for $q$ in the range $(10^4,10^5)$, the sum lies between $2.14 q$ and $3.95 q$; 
for $q$ in the range $(10^5,10^6)$, the sum lies between $2.09 q$ and $3.82 q$; and 
for $q$ in the range $(10^6,10^7)$, the sum lies between $2.10 q$ and $3.77 q$.
Note that as $q$ ranges through these successive intervals, the upper bound on
$1/q$ times the sum of the relative conductors decreases; this is why we are
tempted to suspect that the sum of the relative conductors is~$O(q)$.

\subsection{The probability that a principally-polarized abelian surface is split}
\label{ssec-data-probsplit}

If $S$ is a finite collection of geometric objects having finite automorphism 
groups, we define the \emph{weighted cardinality} $\#' S$ of $S$ by
\[
\#' S = \sum_{s \in S} \oneover{\# \aut s}.
\]
It is well-known that the weighted cardinality can lead to cleaner formulas 
than the usual cardinality. For instance, the weighted cardinality of the set 
of genus-$2$ curves over $\fq$ is equal to $q^3$
(\cite[Prop.~7.1, p.~87]{BrockGranville2001}). A principally-polarized abelian 
surface over a field is either a Jacobian, a product of polarized elliptic 
curves, or the restriction of scalars of a polarized elliptic curve over a 
quadratic extension of the base field 
(\cite[Thm.~3.1, p.~270]{GonzalezGuardiaEtAl2005}). One can show that the
weighted cardinality of the set of products of polarized elliptic curves over 
$\fq$ is $q^2/2$, as is the weighted cardinality of the set of restrictions of
scalars. Thus, if we let $\cala_2$ denote the moduli stack of 
principally-polarized abelian surfaces, then $\#'\cala_2(\fq) = q^3 + q^2$.

For each prime power $q$ we let
\[
c_q = \frac{\sqrt{q} \cdot \#'\cala_{2,\ssplit}(\fq)}{\#'\cala_2(\fq)}
    = \frac{\sqrt{q} \cdot \#'\cala_{2,\ssplit}(\fq)}{q^3 + q^2}.
\]    
For all primes $q<300$ and for $q = 521$ we computed the exact value of $c_q$
by direct enumeration of curves and computation of zeta functions. For 
$q\in \st{1031, 2053, 4099, 16411, 65537}$ (the smallest primes greater 
than $2^i$ for $i=10, 11, 12, 14, 16$) we computed approximations to $c_q$ by 
randomly sampling genus-$2$ curves (with probability inversely proportional to 
their automorphism groups), and then adjusting the probabilities to account for
the non-Jacobians.  We computed enough examples for each of these $q$ to 
determine $c_q$ with a standard deviation of less $0.0005$.  The result of the
computations is displayed in Figure~\ref{figprobsplit}; the (almost invisible) 
error bars on the rightmost five data points indicate the standard deviation.
Note that the horizontal axis is $\log \log q$; even so, the graph looks 
sublinear.  This encourages us to speculate that perhaps the $c_q$ are bounded
away from $0$ and $\infty$.

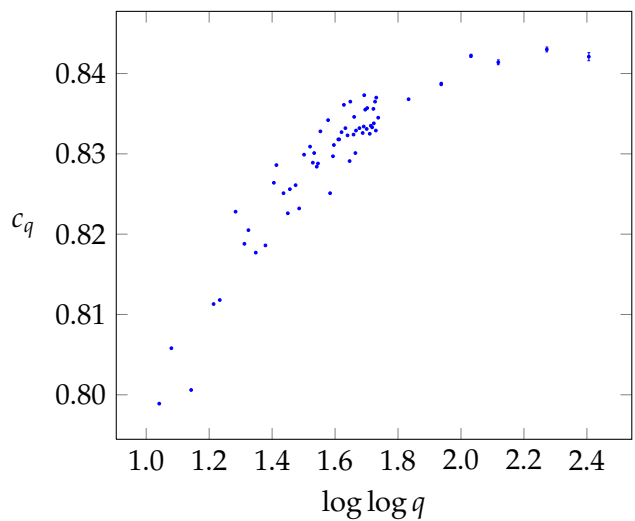
\begin{figure}
\centering
\begin{tikzpicture}
   \begin{axis}[xlabel=$\log \log q$,
                ylabel=$c_q$, 
                ylabel style={rotate=270},
                xtick={1, 1.2, 1.4, 1.6, 1.8, 2, 2.2, 2.4},
                xticklabels={$1.0$,$1.2$,$1.4$,$1.6$,$1.8$,$2.0$,$2.2$,$2.4$},
                ytick={0.8, 0.81, 0.82, 0.83, 0.84},
                yticklabels={$0.80$,$0.81$,$0.82$,$0.83$,$0.84$},
                mark size={0.1ex}]
   \addplot[color=blue, only marks,
            error bars/.cd,
            y dir=both, y explicit,
             ]
     coordinates {
( 1.0414, 0.7989)
( 1.0799, 0.8058)
( 1.1428, 0.8006)
( 1.2141, 0.8113)
( 1.2337, 0.8118)
( 1.2840, 0.8228)
( 1.3120, 0.8188)
( 1.3247, 0.8205)
( 1.3481, 0.8177)
( 1.3788, 0.8186)
( 1.4055, 0.8264)
( 1.4136, 0.8286)
( 1.4362, 0.8251)
( 1.4499, 0.8226)
( 1.4564, 0.8256)
( 1.4746, 0.8261)
( 1.4859, 0.8232)
( 1.5015, 0.8299)
( 1.5205, 0.8309)
( 1.5293, 0.8289)
( 1.5336, 0.8301)
( 1.5418, 0.8284)
( 1.5457, 0.8288)
( 1.5534, 0.8328)
( 1.5778, 0.8342)
( 1.5842, 0.8251)
( 1.5933, 0.8297)
( 1.5962, 0.8311)
( 1.6102, 0.8318)
( 1.6129, 0.8318)
( 1.6206, 0.8327)
( 1.6280, 0.8361)
( 1.6328, 0.8332)
( 1.6396, 0.8323)
( 1.6462, 0.8291)
( 1.6484, 0.8365)
( 1.6587, 0.8324)
( 1.6606, 0.8346)
( 1.6645, 0.8301)
( 1.6664, 0.8329)
( 1.6774, 0.8332)
( 1.6877, 0.8326)
( 1.6910, 0.8334)
( 1.6926, 0.8373)
( 1.6958, 0.8355)
( 1.7005, 0.8331)
( 1.7020, 0.8357)
( 1.7101, 0.8325)
( 1.7136, 0.8335)
( 1.7178, 0.8333)
( 1.7218, 0.8356)
( 1.7231, 0.8338)
( 1.7270, 0.8365)
( 1.7296, 0.8329)
( 1.7308, 0.8370)
( 1.7370, 0.8345)
( 1.8335, 0.8368)
( 1.9371, 0.8387)  +- (0,0.0002)
( 2.0317, 0.8422)  +- (0,0.0002)
( 2.1185, 0.8414)  +- (0,0.0003)
( 2.2727, 0.8430)  +- (0,0.0003)
( 2.4061, 0.8421)  +- (0,0.0005)
        };
  \end{axis}
\end{tikzpicture}
\caption{The values of $c_q$ for the primes $q$ with $17\le q \le 293$, 
         together with $q\in \st{521, 1031, 2053, 4099, 16411, 65537}$. 
         The values of $c_q$ for the five largest $q$ were computed
         experimentally; the error bars indicate one standard deviation.}
\label{figprobsplit}
\end{figure}

\subsection{Reductions of a fixed surface}
\label{ssec-data-reduction}

Let $A/K$ be a principally-polarizable abelian surface over a number field such 
that the absolute endomorphism ring $\End_{\bar K} A$ is isomorphic to $\integ$,
and recall the counting function $\pi_\ssplit(A/K,z)$ introduced 
Section~\ref{secintro}. Conjecture~\ref{conjabsurf} states that  
$\pi_\ssplit(A/K,z)\sim C_A \sqrt z/\log z$; we tested this against actual data
on the splitting behavior of a particular surface $A$ over~$\rat$.

Let $A$ be the Jacobian of the curve over $\rat$ with affine model 
$y^2 = x^5 + x + 6$. Using the methods of~\cite{HarveySutherland2014}, Andrew 
Sutherland computed for us the primes $p<2^{30}$ for which the mod-$p$ reduction
of $A$ is split, thereby giving us the exact value of $\pi_\ssplit(A/\rat,z)$
for all $z\le 2^{30}$. We numerically fit curves of the form 
$a\sqrt{z}/(\log z)^b$ and of the form $c \sqrt z/\log z$ to this function.
For curves of the form $a\sqrt{z}/(\log z)^b$, the best-fitting exponent $b$
was $b \approx 1.02269$, reasonably close to our conjectural value of~$1$. 
For curves of the form $c \sqrt z/\log z$, the best-fitting constant $c$ was
$c \approx 4.4651$. In Figure~\ref{figbig} we present the actual data (in blue)
alongside the best-fitting function $c \sqrt z/\log z$ (in red); the figure 
shows that the idealized function is in close agreement with the actual 
function.

\begin{figure}
\centering
\begin{tikzpicture}
   \begin{axis}[width=0.7\hsize,
                xlabel=$z$,
                ylabel=${\pi_\ssplit(A/\rat,z)}$, 
                ylabel style={rotate=270},
                ylabel shift={1ex},
                scaled ticks=false,
                mark size={0.05ex}]
   \addplot[color=blue]
     coordinates {
(         3,   1) (       149,  10) (       829,  19) (      1877,  28) (      4157,  37) (      5741,  46) (     12203,  55) (     19447,  64) (     23131,  73) (     34763,  82) (     41231,  91) (     52501, 100) (     64301, 109) (     74597, 118) (     84467, 127) (    104393, 136) (    122011, 145) 
(    148501, 154) (    165719, 163) (    195743, 172) (    216061, 181) (    239753, 190) (    266587, 199) (    285289, 208) (    316817, 217) (    343169, 226) (    368287, 235) (    398407, 244) (    435779, 253) (    486589, 262) (    502441, 271) (    562297, 280) (    618329, 289) (    675743, 298) 
(    740011, 307) (    801841, 316) (    861659, 325) (    920509, 334) (    999437, 343) (   1065133, 352) (   1151243, 361) (   1225109, 370) (   1288709, 379) (   1362511, 388) (   1487917, 397) (   1570007, 406) (   1640761, 415) (   1738493, 424) (   1803743, 433) (   1929467, 442) (   2053439, 451) 
(   2164807, 460) (   2227429, 469) (   2379929, 478) (   2499137, 487) (   2697151, 496) (   2816213, 505) (   3014051, 514) (   3119491, 523) (   3241559, 532) (   3383801, 541) (   3451681, 550) (   3623071, 559) (   3809171, 568) (   3958621, 577) (   4059619, 586) (   4261009, 595) (   4414313, 604) 
(   4623659, 613) (   4726069, 622) (   4851883, 631) (   5029411, 640) (   5198377, 649) (   5495299, 658) (   5635717, 667) (   5790857, 676) (   5908891, 685) (   6134533, 694) (   6258107, 703) (   6450049, 712) (   6780551, 721) (   6975091, 730) (   7168873, 739) (   7333049, 748) (   7453487, 757) 
(   7537003, 766) (   7886111, 775) (   8033591, 784) (   8186291, 793) (   8422571, 802) (   8592533, 811) (   8834627, 820) (   8995757, 829) (   9283171, 838) (   9444217, 847) (   9643813, 856) (   9907081, 865) (  10104139, 874) (  10276979, 883) (  10555553, 892) (  10683773, 901) (  10808939, 910) 
(  11116151, 919) (  11406401, 928) (  11795677, 937) (  12058477, 946) (  12205799, 955) (  12518399, 964) (  12760399, 973) (  12997441, 982) (  13349807, 991) (  13652209,1000) (  13951279,1009) (  14203139,1018) (  14383687,1027) (  14558543,1036) (  15039793,1045) (  15181207,1054) (  15341393,1063) 
(  15573653,1072) (  15923461,1081) (  16223939,1090) (  16546513,1099) (  16913009,1108) (  17177011,1117) (  17417689,1126) (  17612291,1135) (  17843083,1144) (  18268447,1153) (  18636119,1162) (  19059503,1171) (  19171013,1180) (  19635631,1189) (  19971563,1198) (  20217787,1207) (  20572781,1216) 
(  21052159,1225) (  21387851,1234) (  21721031,1243) (  22093843,1252) (  22513901,1261) (  22960799,1270) (  23272687,1279) (  23573569,1288) (  23972153,1297) (  24346897,1306) (  24950537,1315) (  25534573,1324) (  25827469,1333) (  26179691,1342) (  26483503,1351) (  26834771,1360) (  27130289,1369) 
(  27594613,1378) (  27910723,1387) (  28368511,1396) (  28645367,1405) (  29070397,1414) (  29430637,1423) (  29768113,1432) (  30474887,1441) (  30855211,1450) (  31213961,1459) (  31667551,1468) (  31913839,1477) (  32589911,1486) (  32926843,1495) (  33489791,1504) (  33782473,1513) (  34324187,1522) 
(  34836107,1531) (  35308543,1540) (  35712559,1549) (  36221329,1558) (  36742691,1567) (  37156907,1576) (  37467799,1585) (  37748363,1594) (  38166061,1603) (  38815363,1612) (  39284389,1621) (  39732419,1630) (  40107751,1639) (  40391101,1648) (  40689421,1657) (  41761003,1666) (  42246319,1675) 
(  42454897,1684) (  43223497,1693) (  43700561,1702) (  44175281,1711) (  44630479,1720) (  45249877,1729) (  45753481,1738) (  46176643,1747) (  46801031,1756) (  47284417,1765) (  48275021,1774) (  49072451,1783) (  49790353,1792) (  50363609,1801) (  50795461,1810) (  51300853,1819) (  51894131,1828) 
(  52300189,1837) (  52644743,1846) (  53047559,1855) (  53565191,1864) (  54087833,1873) (  54669661,1882) (  55269787,1891) (  55915099,1900) (  56563099,1909) (  57155573,1918) (  57813359,1927) (  58397231,1936) (  59199037,1945) (  59741119,1954) (  60410267,1963) (  61107317,1972) (  61499377,1981) 
(  62462689,1990) (  62873933,1999) (  64361903,2008) (  64930897,2017) (  65542193,2026) (  66103727,2035) (  66582689,2044) (  67838591,2053) (  68467571,2062) (  69105373,2071) (  70060577,2080) (  70746757,2089) (  71210473,2098) (  71919293,2107) (  72751639,2116) (  73269557,2125) (  73960343,2134) 
(  74707081,2143) (  75489179,2152) (  76505719,2161) (  77490529,2170) (  78513737,2179) (  79083373,2188) (  79590613,2197) (  80430809,2206) (  81450713,2215) (  82507913,2224) (  83110697,2233) (  83764169,2242) (  85068457,2251) (  85906757,2260) (  86566367,2269) (  87614047,2278) (  88066607,2287) 
(  89088893,2296) (  89639437,2305) (  90817999,2314) (  91280303,2323) (  91755607,2332) (  92686903,2341) (  93343259,2350) (  94028569,2359) (  94766059,2368) (  95810437,2377) (  96709441,2386) (  97470991,2395) (  98511337,2404) (  99624331,2413) ( 100651939,2422) ( 101451001,2431) ( 102505153,2440) 
( 103406783,2449) ( 104159009,2458) ( 105596479,2467) ( 106388587,2476) ( 107316553,2485) ( 108718507,2494) ( 110546207,2503) ( 111518633,2512) ( 111995921,2521) ( 112835659,2530) ( 113597579,2539) ( 114538097,2548) ( 115223473,2557) ( 116077537,2566) ( 117301529,2575) ( 117760259,2584) ( 119081393,2593) 
( 119822009,2602) ( 120589061,2611) ( 121314283,2620) ( 122324621,2629) ( 123180367,2638) ( 124052207,2647) ( 124948723,2656) ( 125845129,2665) ( 126698717,2674) ( 127512443,2683) ( 128441791,2692) ( 129486043,2701) ( 130318127,2710) ( 131346779,2719) ( 132178507,2728) ( 133307323,2737) ( 135000709,2746) 
( 135375803,2755) ( 136238807,2764) ( 137089597,2773) ( 137709331,2782) ( 138747467,2791) ( 140213473,2800) ( 141230987,2809) ( 142220383,2818) ( 143286037,2827) ( 143981881,2836) ( 144613837,2845) ( 145406663,2854) ( 146511539,2863) ( 147344941,2872) ( 148304071,2881) ( 149502943,2890) ( 150277187,2899) 
( 151079881,2908) ( 152614537,2917) ( 153523157,2926) ( 154753427,2935) ( 155887741,2944) ( 156485851,2953) ( 157052743,2962) ( 158161307,2971) ( 159094973,2980) ( 160163023,2989) ( 160818773,2998) ( 162187423,3007) ( 162779527,3016) ( 163688123,3025) ( 164719697,3034) ( 165913703,3043) ( 167013443,3052) 
( 167632363,3061) ( 169211687,3070) ( 170571179,3079) ( 172062001,3088) ( 173043929,3097) ( 174460969,3106) ( 175623191,3115) ( 176719303,3124) ( 177980287,3133) ( 179247799,3142) ( 180217963,3151) ( 181644647,3160) ( 182591657,3169) ( 184230413,3178) ( 185737073,3187) ( 186701447,3196) ( 187689347,3205) 
( 189350009,3214) ( 190707893,3223) ( 192066059,3232) ( 193469009,3241) ( 194015399,3250) ( 195229283,3259) ( 196497209,3268) ( 197340497,3277) ( 198615211,3286) ( 199458421,3295) ( 200498791,3304) ( 201280943,3313) ( 203534711,3322) ( 205077709,3331) ( 206530561,3340) ( 207823333,3349) ( 208854823,3358) 
( 210272399,3367) ( 211135789,3376) ( 212135051,3385) ( 212703653,3394) ( 213467627,3403) ( 214528637,3412) ( 215481187,3421) ( 216582419,3430) ( 217816063,3439) ( 219048439,3448) ( 220730491,3457) ( 222464701,3466) ( 223150981,3475) ( 224150299,3484) ( 225623927,3493) ( 227066291,3502) ( 228428729,3511) 
( 229919351,3520) ( 230839127,3529) ( 232680667,3538) ( 233371891,3547) ( 234809447,3556) ( 235576717,3565) ( 236075713,3574) ( 237482459,3583) ( 239137471,3592) ( 240936823,3601) ( 242003389,3610) ( 242948749,3619) ( 243317873,3628) ( 244379977,3637) ( 246559813,3646) ( 247895579,3655) ( 248580047,3664) 
( 249859609,3673) ( 252220091,3682) ( 253094581,3691) ( 254073763,3700) ( 254969567,3709) ( 257021209,3718) ( 257706913,3727) ( 259006691,3736) ( 259692337,3745) ( 261462917,3754) ( 262447631,3763) ( 263777039,3772) ( 265725179,3781) ( 267252959,3790) ( 269235271,3799) ( 270702161,3808) ( 271848007,3817) 
( 272809909,3826) ( 274054811,3835) ( 276292199,3844) ( 278061131,3853) ( 279545111,3862) ( 280760479,3871) ( 282548621,3880) ( 284452859,3889) ( 285590443,3898) ( 286797451,3907) ( 287631529,3916) ( 289232761,3925) ( 291180733,3934) ( 294235589,3943) ( 295445963,3952) ( 296667407,3961) ( 299440909,3970) 
( 301477447,3979) ( 302352047,3988) ( 304761857,3997) ( 306752363,4006) ( 308976389,4015) ( 309747853,4024) ( 311412053,4033) ( 313641919,4042) ( 315414613,4051) ( 316502567,4060) ( 318760649,4069) ( 320058839,4078) ( 321216349,4087) ( 323076401,4096) ( 324473153,4105) ( 325606003,4114) ( 327545243,4123) 
( 328783561,4132) ( 330096953,4141) ( 331338727,4150) ( 334028693,4159) ( 335683373,4168) ( 337622081,4177) ( 339361843,4186) ( 340641731,4195) ( 341680727,4204) ( 342489817,4213) ( 344111381,4222) ( 345455447,4231) ( 348627809,4240) ( 351007763,4249) ( 351596257,4258) ( 353630531,4267) ( 354308099,4276) 
( 355683649,4285) ( 357708853,4294) ( 359524541,4303) ( 360775879,4312) ( 361787851,4321) ( 363130219,4330) ( 365412083,4339) ( 367251449,4348) ( 370801481,4357) ( 372348029,4366) ( 374261011,4375) ( 376390163,4384) ( 377765669,4393) ( 379945609,4402) ( 380794619,4411) ( 382950433,4420) ( 384220769,4429) 
( 386930683,4438) ( 388338193,4447) ( 391237661,4456) ( 391846547,4465) ( 392892641,4474) ( 395608153,4483) ( 397487813,4492) ( 399833789,4501) ( 400815509,4510) ( 402206291,4519) ( 403991891,4528) ( 405355061,4537) ( 406986967,4546) ( 409712117,4555) ( 411342367,4564) ( 412640869,4573) ( 414224983,4582) 
( 416323751,4591) ( 418192903,4600) ( 419459611,4609) ( 420936671,4618) ( 423096211,4627) ( 424376549,4636) ( 425936939,4645) ( 428068673,4654) ( 429573253,4663) ( 432292901,4672) ( 435236773,4681) ( 437151479,4690) ( 439093019,4699) ( 441057247,4708) ( 444043429,4717) ( 446352791,4726) ( 447914693,4735) 
( 450132191,4744) ( 453067031,4753) ( 454482961,4762) ( 455813483,4771) ( 458595523,4780) ( 460533643,4789) ( 462972001,4798) ( 464661859,4807) ( 466116649,4816) ( 468437491,4825) ( 470335573,4834) ( 472971019,4843) ( 474541381,4852) ( 475768109,4861) ( 477815573,4870) ( 479650799,4879) ( 482011571,4888) 
( 484623661,4897) ( 487132651,4906) ( 490360177,4915) ( 492858269,4924) ( 495374779,4933) ( 498052369,4942) ( 499634683,4951) ( 502768271,4960) ( 505495007,4969) ( 507563083,4978) ( 509364029,4987) ( 511791191,4996) ( 513070783,5005) ( 514083799,5014) ( 515973119,5023) ( 517793323,5032) ( 519456107,5041) 
( 520353481,5050) ( 523492337,5059) ( 525317621,5068) ( 527059889,5077) ( 529359679,5086) ( 531521357,5095) ( 533419363,5104) ( 535365373,5113) ( 537355073,5122) ( 539409427,5131) ( 541081631,5140) ( 543460481,5149) ( 546469711,5158) ( 548307857,5167) ( 549854597,5176) ( 551141081,5185) ( 553000309,5194) 
( 555127709,5203) ( 556955081,5212) ( 558860087,5221) ( 560611609,5230) ( 562391861,5239) ( 564010267,5248) ( 565268339,5257) ( 567420341,5266) ( 570222893,5275) ( 572158771,5284) ( 575842529,5293) ( 578374019,5302) ( 579532529,5311) ( 581386807,5320) ( 584655977,5329) ( 587222443,5338) ( 589781053,5347) 
( 591096911,5356) ( 593164303,5365) ( 595493567,5374) ( 597756727,5383) ( 599140081,5392) ( 601248961,5401) ( 603665771,5410) ( 606325033,5419) ( 610204159,5428) ( 612733097,5437) ( 614527211,5446) ( 616026809,5455) ( 618602041,5464) ( 621979591,5473) ( 625395313,5482) ( 627196819,5491) ( 629397479,5500) 
( 630695647,5509) ( 632961367,5518) ( 636555977,5527) ( 638406851,5536) ( 640674709,5545) ( 641645597,5554) ( 643530149,5563) ( 646196269,5572) ( 648992339,5581) ( 650172409,5590) ( 652366933,5599) ( 655619449,5608) ( 658609571,5617) ( 661628983,5626) ( 663132689,5635) ( 665186803,5644) ( 668025401,5653) 
( 669881489,5662) ( 671396501,5671) ( 673824343,5680) ( 678718427,5689) ( 680758223,5698) ( 683014811,5707) ( 685462153,5716) ( 687180829,5725) ( 689305391,5734) ( 691816507,5743) ( 693706399,5752) ( 695221757,5761) ( 698915699,5770) ( 700295251,5779) ( 703620553,5788) ( 705899477,5797) ( 707358433,5806) 
( 708672799,5815) ( 709421459,5824) ( 711027019,5833) ( 713183377,5842) ( 715744703,5851) ( 718497907,5860) ( 720049849,5869) ( 722188367,5878) ( 724695403,5887) ( 726326369,5896) ( 728655563,5905) ( 730159699,5914) ( 734358563,5923) ( 737966353,5932) ( 740529917,5941) ( 742315633,5950) ( 743851243,5959) 
( 746926709,5968) ( 749380507,5977) ( 751283629,5986) ( 753771521,5995) ( 754876081,6004) ( 757297613,6013) ( 762032951,6022) ( 763071457,6031) ( 764368831,6040) ( 765556711,6049) ( 767614313,6058) ( 772377773,6067) ( 775583213,6076) ( 777277561,6085) ( 778350961,6094) ( 781385753,6103) ( 784130749,6112) 
( 785139167,6121) ( 787340539,6130) ( 790802437,6139) ( 792726811,6148) ( 793882457,6157) ( 796955933,6166) ( 799249819,6175) ( 801802961,6184) ( 804694069,6193) ( 808127773,6202) ( 809402621,6211) ( 811785053,6220) ( 814420681,6229) ( 818188601,6238) ( 822383171,6247) ( 824646397,6256) ( 827384641,6265) 
( 829991111,6274) ( 832498871,6283) ( 835673863,6292) ( 839008171,6301) ( 841084021,6310) ( 844369091,6319) ( 845890733,6328) ( 848861903,6337) ( 850656349,6346) ( 853981643,6355) ( 855630511,6364) ( 857130341,6373) ( 860443121,6382) ( 862613417,6391) ( 865982587,6400) ( 870444793,6409) ( 871889779,6418) 
( 874189493,6427) ( 875892053,6436) ( 878514269,6445) ( 880474501,6454) ( 882587093,6463) ( 886604993,6472) ( 889137649,6481) ( 892435601,6490) ( 893808967,6499) ( 896407583,6508) ( 899729813,6517) ( 902861929,6526) ( 905387993,6535) ( 907794287,6544) ( 910585699,6553) ( 913083707,6562) ( 917461109,6571) 
( 921388399,6580) ( 925032463,6589) ( 928081139,6598) ( 932242433,6607) ( 936849373,6616) ( 938912099,6625) ( 940335503,6634) ( 943239527,6643) ( 945196529,6652) ( 947990179,6661) ( 949532737,6670) ( 953016577,6679) ( 955435727,6688) ( 959079977,6697) ( 962075381,6706) ( 964541069,6715) ( 966330223,6724) 
( 968801209,6733) ( 971304133,6742) ( 974622581,6751) ( 976976779,6760) ( 979139891,6769) ( 982061053,6778) ( 985066801,6787) ( 986324483,6796) ( 987583813,6805) ( 990096301,6814) ( 993499093,6823) ( 995519017,6832) ( 998563057,6841) (1001945641,6850) (1004478637,6859) (1006010441,6868) (1009070441,6877) 
(1012055897,6886) (1015503277,6895) (1020191903,6904) (1023926747,6913) (1028142769,6922) (1032911051,6931) (1035864449,6940) (1038731801,6949) (1043088911,6958) (1046756663,6967) (1049204771,6976) (1051851467,6985) (1056440309,6994) (1059501661,7003) (1064376913,7012) (1069839391,7021) (1073360719,7030) 
        };
   \addplot[color=red]
     coordinates {
(         3,   7) (       149,  11) (       829,  19) (      1877,  26) (      4157,  35) (      5741,  39) (     12203,  52) (     19447,  63) (     23131,  68) (     34763,  80) (     41231,  85) (     52501,  94) (     64301, 102) (     74597, 109) (     84467, 114) (    104393, 125) (    122011, 133) 
(    148501, 144) (    165719, 151) (    195743, 162) (    216061, 169) (    239753, 176) (    266587, 185) (    285289, 190) (    316817, 198) (    343169, 205) (    368287, 211) (    398407, 219) (    435779, 227) (    486589, 238) (    502441, 241) (    562297, 253) (    618329, 263) (    675743, 273) 
(    740011, 284) (    801841, 294) (    861659, 303) (    920509, 312) (    999437, 323) (   1065133, 332) (   1151243, 343) (   1225109, 353) (   1288709, 360) (   1362511, 369) (   1487917, 383) (   1570007, 392) (   1640761, 400) (   1738493, 410) (   1803743, 416) (   1929467, 429) (   2053439, 440) 
(   2164807, 450) (   2227429, 456) (   2379929, 469) (   2499137, 479) (   2697151, 495) (   2816213, 505) (   3014051, 520) (   3119491, 527) (   3241559, 536) (   3383801, 546) (   3451681, 551) (   3623071, 563) (   3809171, 575) (   3958621, 585) (   4059619, 591) (   4261009, 604) (   4414313, 613) 
(   4623659, 626) (   4726069, 632) (   4851883, 639) (   5029411, 649) (   5198377, 658) (   5495299, 674) (   5635717, 682) (   5790857, 690) (   5908891, 696) (   6134533, 708) (   6258107, 714) (   6450049, 723) (   6780551, 739) (   6975091, 748) (   7168873, 757) (   7333049, 765) (   7453487, 770) 
(   7537003, 774) (   7886111, 790) (   8033591, 796) (   8186291, 803) (   8422571, 813) (   8592533, 820) (   8834627, 830) (   8995757, 836) (   9283171, 848) (   9444217, 854) (   9643813, 862) (   9907081, 872) (  10104139, 880) (  10276979, 887) (  10555553, 897) (  10683773, 902) (  10808939, 906) 
(  11116151, 918) (  11406401, 928) (  11795677, 942) (  12058477, 951) (  12205799, 956) (  12518399, 967) (  12760399, 975) (  12997441, 983) (  13349807, 994) (  13652209,1004) (  13951279,1014) (  14203139,1022) (  14383687,1027) (  14558543,1033) (  15039793,1048) (  15181207,1052) (  15341393,1057) 
(  15573653,1064) (  15923461,1074) (  16223939,1083) (  16546513,1093) (  16913009,1103) (  17177011,1111) (  17417689,1118) (  17612291,1123) (  17843083,1130) (  18268447,1141) (  18636119,1151) (  19059503,1163) (  19171013,1166) (  19635631,1178) (  19971563,1187) (  20217787,1193) (  20572781,1203) 
(  21052159,1215) (  21387851,1223) (  21721031,1232) (  22093843,1241) (  22513901,1251) (  22960799,1262) (  23272687,1270) (  23573569,1277) (  23972153,1287) (  24346897,1295) (  24950537,1309) (  25534573,1323) (  25827469,1330) (  26179691,1338) (  26483503,1344) (  26834771,1352) (  27130289,1359) 
(  27594613,1369) (  27910723,1376) (  28368511,1386) (  28645367,1392) (  29070397,1401) (  29430637,1409) (  29768113,1416) (  30474887,1430) (  30855211,1438) (  31213961,1446) (  31667551,1455) (  31913839,1460) (  32589911,1473) (  32926843,1480) (  33489791,1491) (  33782473,1497) (  34324187,1508) 
(  34836107,1518) (  35308543,1527) (  35712559,1534) (  36221329,1544) (  36742691,1554) (  37156907,1561) (  37467799,1567) (  37748363,1572) (  38166061,1580) (  38815363,1592) (  39284389,1600) (  39732419,1609) (  40107751,1615) (  40391101,1620) (  40689421,1626) (  41761003,1644) (  42246319,1653) 
(  42454897,1656) (  43223497,1670) (  43700561,1678) (  44175281,1686) (  44630479,1694) (  45249877,1704) (  45753481,1712) (  46176643,1719) (  46801031,1730) (  47284417,1737) (  48275021,1754) (  49072451,1766) (  49790353,1778) (  50363609,1787) (  50795461,1794) (  51300853,1801) (  51894131,1811) 
(  52300189,1817) (  52644743,1822) (  53047559,1828) (  53565191,1836) (  54087833,1844) (  54669661,1853) (  55269787,1862) (  55915099,1872) (  56563099,1881) (  57155573,1890) (  57813359,1900) (  58397231,1908) (  59199037,1920) (  59741119,1927) (  60410267,1937) (  61107317,1947) (  61499377,1952) 
(  62462689,1966) (  62873933,1972) (  64361903,1992) (  64930897,2000) (  65542193,2008) (  66103727,2016) (  66582689,2023) (  67838591,2039) (  68467571,2048) (  69105373,2056) (  70060577,2069) (  70746757,2078) (  71210473,2084) (  71919293,2093) (  72751639,2104) (  73269557,2111) (  73960343,2119) 
(  74707081,2129) (  75489179,2139) (  76505719,2151) (  77490529,2164) (  78513737,2176) (  79083373,2183) (  79590613,2190) (  80430809,2200) (  81450713,2212) (  82507913,2225) (  83110697,2232) (  83764169,2240) (  85068457,2255) (  85906757,2265) (  86566367,2273) (  87614047,2285) (  88066607,2291) 
(  89088893,2302) (  89639437,2309) (  90817999,2322) (  91280303,2327) (  91755607,2333) (  92686903,2343) (  93343259,2351) (  94028569,2358) (  94766059,2367) (  95810437,2378) (  96709441,2388) (  97470991,2396) (  98511337,2408) (  99624331,2420) ( 100651939,2431) ( 101451001,2440) ( 102505153,2451) 
( 103406783,2460) ( 104159009,2468) ( 105596479,2484) ( 106388587,2492) ( 107316553,2501) ( 108718507,2516) ( 110546207,2535) ( 111518633,2545) ( 111995921,2550) ( 112835659,2558) ( 113597579,2566) ( 114538097,2575) ( 115223473,2582) ( 116077537,2591) ( 117301529,2603) ( 117760259,2607) ( 119081393,2620) 
( 119822009,2628) ( 120589061,2635) ( 121314283,2642) ( 122324621,2652) ( 123180367,2660) ( 124052207,2669) ( 124948723,2677) ( 125845129,2686) ( 126698717,2694) ( 127512443,2702) ( 128441791,2710) ( 129486043,2720) ( 130318127,2728) ( 131346779,2738) ( 132178507,2745) ( 133307323,2756) ( 135000709,2771) 
( 135375803,2775) ( 136238807,2783) ( 137089597,2790) ( 137709331,2796) ( 138747467,2805) ( 140213473,2819) ( 141230987,2828) ( 142220383,2837) ( 143286037,2846) ( 143981881,2852) ( 144613837,2858) ( 145406663,2865) ( 146511539,2874) ( 147344941,2882) ( 148304071,2890) ( 149502943,2901) ( 150277187,2907) 
( 151079881,2914) ( 152614537,2927) ( 153523157,2935) ( 154753427,2946) ( 155887741,2955) ( 156485851,2960) ( 157052743,2965) ( 158161307,2974) ( 159094973,2982) ( 160163023,2991) ( 160818773,2997) ( 162187423,3008) ( 162779527,3013) ( 163688123,3020) ( 164719697,3029) ( 165913703,3039) ( 167013443,3048) 
( 167632363,3053) ( 169211687,3066) ( 170571179,3077) ( 172062001,3089) ( 173043929,3096) ( 174460969,3108) ( 175623191,3117) ( 176719303,3126) ( 177980287,3136) ( 179247799,3146) ( 180217963,3153) ( 181644647,3164) ( 182591657,3172) ( 184230413,3184) ( 185737073,3196) ( 186701447,3204) ( 187689347,3211) 
( 189350009,3224) ( 190707893,3234) ( 192066059,3244) ( 193469009,3255) ( 194015399,3259) ( 195229283,3268) ( 196497209,3278) ( 197340497,3284) ( 198615211,3293) ( 199458421,3300) ( 200498791,3307) ( 201280943,3313) ( 203534711,3330) ( 205077709,3341) ( 206530561,3352) ( 207823333,3361) ( 208854823,3368) 
( 210272399,3379) ( 211135789,3385) ( 212135051,3392) ( 212703653,3396) ( 213467627,3402) ( 214528637,3409) ( 215481187,3416) ( 216582419,3424) ( 217816063,3432) ( 219048439,3441) ( 220730491,3453) ( 222464701,3465) ( 223150981,3470) ( 224150299,3477) ( 225623927,3487) ( 227066291,3497) ( 228428729,3506) 
( 229919351,3517) ( 230839127,3523) ( 232680667,3535) ( 233371891,3540) ( 234809447,3550) ( 235576717,3555) ( 236075713,3558) ( 237482459,3568) ( 239137471,3579) ( 240936823,3591) ( 242003389,3598) ( 242948749,3605) ( 243317873,3607) ( 244379977,3614) ( 246559813,3628) ( 247895579,3637) ( 248580047,3642) 
( 249859609,3650) ( 252220091,3666) ( 253094581,3671) ( 254073763,3678) ( 254969567,3683) ( 257021209,3697) ( 257706913,3701) ( 259006691,3709) ( 259692337,3714) ( 261462917,3725) ( 262447631,3731) ( 263777039,3740) ( 265725179,3752) ( 267252959,3762) ( 269235271,3774) ( 270702161,3784) ( 271848007,3791) 
( 272809909,3797) ( 274054811,3805) ( 276292199,3818) ( 278061131,3829) ( 279545111,3839) ( 280760479,3846) ( 282548621,3857) ( 284452859,3869) ( 285590443,3876) ( 286797451,3883) ( 287631529,3888) ( 289232761,3898) ( 291180733,3909) ( 294235589,3928) ( 295445963,3935) ( 296667407,3942) ( 299440909,3959) 
( 301477447,3971) ( 302352047,3976) ( 304761857,3990) ( 306752363,4002) ( 308976389,4015) ( 309747853,4019) ( 311412053,4029) ( 313641919,4042) ( 315414613,4052) ( 316502567,4059) ( 318760649,4072) ( 320058839,4079) ( 321216349,4086) ( 323076401,4096) ( 324473153,4104) ( 325606003,4111) ( 327545243,4122) 
( 328783561,4128) ( 330096953,4136) ( 331338727,4143) ( 334028693,4158) ( 335683373,4167) ( 337622081,4178) ( 339361843,4188) ( 340641731,4195) ( 341680727,4200) ( 342489817,4205) ( 344111381,4214) ( 345455447,4221) ( 348627809,4239) ( 351007763,4252) ( 351596257,4255) ( 353630531,4266) ( 354308099,4269) 
( 355683649,4277) ( 357708853,4288) ( 359524541,4298) ( 360775879,4304) ( 361787851,4310) ( 363130219,4317) ( 365412083,4329) ( 367251449,4339) ( 370801481,4358) ( 372348029,4366) ( 374261011,4376) ( 376390163,4387) ( 377765669,4394) ( 379945609,4406) ( 380794619,4410) ( 382950433,4421) ( 384220769,4428) 
( 386930683,4442) ( 388338193,4449) ( 391237661,4464) ( 391846547,4467) ( 392892641,4472) ( 395608153,4486) ( 397487813,4496) ( 399833789,4508) ( 400815509,4513) ( 402206291,4520) ( 403991891,4529) ( 405355061,4536) ( 406986967,4544) ( 409712117,4558) ( 411342367,4566) ( 412640869,4572) ( 414224983,4580) 
( 416323751,4590) ( 418192903,4600) ( 419459611,4606) ( 420936671,4613) ( 423096211,4624) ( 424376549,4630) ( 425936939,4638) ( 428068673,4648) ( 429573253,4656) ( 432292901,4669) ( 435236773,4683) ( 437151479,4692) ( 439093019,4702) ( 441057247,4711) ( 444043429,4725) ( 446352791,4736) ( 447914693,4744) 
( 450132191,4754) ( 453067031,4768) ( 454482961,4775) ( 455813483,4781) ( 458595523,4795) ( 460533643,4804) ( 462972001,4815) ( 464661859,4823) ( 466116649,4830) ( 468437491,4841) ( 470335573,4849) ( 472971019,4862) ( 474541381,4869) ( 475768109,4874) ( 477815573,4884) ( 479650799,4892) ( 482011571,4903) 
( 484623661,4915) ( 487132651,4927) ( 490360177,4941) ( 492858269,4952) ( 495374779,4964) ( 498052369,4976) ( 499634683,4983) ( 502768271,4997) ( 505495007,5009) ( 507563083,5018) ( 509364029,5026) ( 511791191,5037) ( 513070783,5043) ( 514083799,5047) ( 515973119,5056) ( 517793323,5064) ( 519456107,5071) 
( 520353481,5075) ( 523492337,5089) ( 525317621,5097) ( 527059889,5104) ( 529359679,5114) ( 531521357,5124) ( 533419363,5132) ( 535365373,5140) ( 537355073,5149) ( 539409427,5158) ( 541081631,5165) ( 543460481,5175) ( 546469711,5188) ( 548307857,5196) ( 549854597,5203) ( 551141081,5208) ( 553000309,5216) 
( 555127709,5225) ( 556955081,5233) ( 558860087,5241) ( 560611609,5248) ( 562391861,5256) ( 564010267,5262) ( 565268339,5268) ( 567420341,5277) ( 570222893,5289) ( 572158771,5297) ( 575842529,5312) ( 578374019,5322) ( 579532529,5327) ( 581386807,5335) ( 584655977,5348) ( 587222443,5359) ( 589781053,5369) 
( 591096911,5375) ( 593164303,5383) ( 595493567,5393) ( 597756727,5402) ( 599140081,5408) ( 601248961,5416) ( 603665771,5426) ( 606325033,5437) ( 610204159,5452) ( 612733097,5463) ( 614527211,5470) ( 616026809,5476) ( 618602041,5486) ( 621979591,5500) ( 625395313,5513) ( 627196819,5520) ( 629397479,5529) 
( 630695647,5534) ( 632961367,5543) ( 636555977,5557) ( 638406851,5565) ( 640674709,5573) ( 641645597,5577) ( 643530149,5585) ( 646196269,5595) ( 648992339,5606) ( 650172409,5611) ( 652366933,5619) ( 655619449,5632) ( 658609571,5643) ( 661628983,5655) ( 663132689,5661) ( 665186803,5669) ( 668025401,5680) 
( 669881489,5687) ( 671396501,5692) ( 673824343,5702) ( 678718427,5720) ( 680758223,5728) ( 683014811,5737) ( 685462153,5746) ( 687180829,5752) ( 689305391,5760) ( 691816507,5770) ( 693706399,5777) ( 695221757,5783) ( 698915699,5796) ( 700295251,5802) ( 703620553,5814) ( 705899477,5822) ( 707358433,5828) 
( 708672799,5833) ( 709421459,5836) ( 711027019,5842) ( 713183377,5849) ( 715744703,5859) ( 718497907,5869) ( 720049849,5875) ( 722188367,5883) ( 724695403,5892) ( 726326369,5898) ( 728655563,5906) ( 730159699,5912) ( 734358563,5927) ( 737966353,5940) ( 740529917,5950) ( 742315633,5956) ( 743851243,5962) 
( 746926709,5973) ( 749380507,5982) ( 751283629,5988) ( 753771521,5997) ( 754876081,6001) ( 757297613,6010) ( 762032951,6027) ( 763071457,6031) ( 764368831,6035) ( 765556711,6039) ( 767614313,6047) ( 772377773,6064) ( 775583213,6075) ( 777277561,6081) ( 778350961,6085) ( 781385753,6095) ( 784130749,6105) 
( 785139167,6109) ( 787340539,6116) ( 790802437,6129) ( 792726811,6135) ( 793882457,6139) ( 796955933,6150) ( 799249819,6158) ( 801802961,6167) ( 804694069,6177) ( 808127773,6189) ( 809402621,6193) ( 811785053,6201) ( 814420681,6210) ( 818188601,6223) ( 822383171,6238) ( 824646397,6246) ( 827384641,6255) 
( 829991111,6264) ( 832498871,6272) ( 835673863,6283) ( 839008171,6294) ( 841084021,6301) ( 844369091,6313) ( 845890733,6318) ( 848861903,6328) ( 850656349,6334) ( 853981643,6345) ( 855630511,6350) ( 857130341,6355) ( 860443121,6366) ( 862613417,6374) ( 865982587,6385) ( 870444793,6400) ( 871889779,6405) 
( 874189493,6412) ( 875892053,6418) ( 878514269,6426) ( 880474501,6433) ( 882587093,6440) ( 886604993,6453) ( 889137649,6461) ( 892435601,6472) ( 893808967,6477) ( 896407583,6485) ( 899729813,6496) ( 902861929,6506) ( 905387993,6514) ( 907794287,6522) ( 910585699,6531) ( 913083707,6539) ( 917461109,6554) 
( 921388399,6566) ( 925032463,6578) ( 928081139,6588) ( 932242433,6601) ( 936849373,6616) ( 938912099,6622) ( 940335503,6627) ( 943239527,6636) ( 945196529,6642) ( 947990179,6651) ( 949532737,6656) ( 953016577,6667) ( 955435727,6675) ( 959079977,6686) ( 962075381,6696) ( 964541069,6703) ( 966330223,6709) 
( 968801209,6717) ( 971304133,6725) ( 974622581,6735) ( 976976779,6742) ( 979139891,6749) ( 982061053,6758) ( 985066801,6767) ( 986324483,6771) ( 987583813,6775) ( 990096301,6783) ( 993499093,6794) ( 995519017,6800) ( 998563057,6809) (1001945641,6820) (1004478637,6827) (1006010441,6832) (1009070441,6841) 
(1012055897,6851) (1015503277,6861) (1020191903,6875) (1023926747,6887) (1028142769,6900) (1032911051,6914) (1035864449,6923) (1038731801,6932) (1043088911,6945) (1046756663,6956) (1049204771,6963) (1051851467,6971) (1056440309,6985) (1059501661,6994) (1064376913,7008) (1069839391,7025) (1073360719,7035) 
        };
  \end{axis}
\end{tikzpicture}
\caption{The blue curve plots the function $\pi_\ssplit(A/\rat,z)$ 
         for the Jacobian $A$ of the curve $y^2 = x^5 + x + 6$ over $\rat$.
         The red curve is $c \sqrt{z}/\log z$, with $c\approx 4.4651$.}
\label{figbig}
\end{figure}
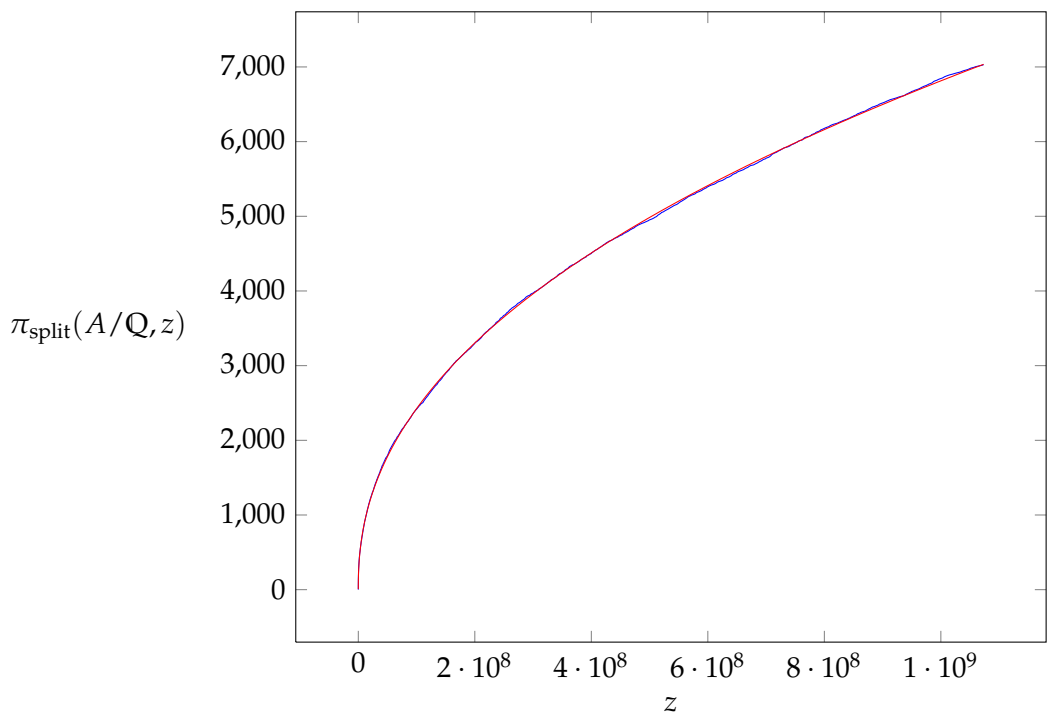

\subsection{Further directions}
\label{ssec-further}

We noted in Section~\ref{secintro} that our definition of 
$\cala_{2,\ssplit}(\fq)$ was perhaps not as natural as it could be --- one 
could also ask about principally-polarized surfaces that split over~$\fqbar$,
not just over $\fq$ itself. We suspect that a result like Theorem~\ref{thmain}
holds for this more general type of splitting. To prove such a theorem, one
would need to estimate the number of principally-polarized surfaces in several 
types of isogeny classes: the simple ordinary isogeny classes that are
geometrically split (which are enumerated in 
\cite[Thm.~6, p.~145]{HoweZhu2002}); and the supersingular isogeny classes 
(which are all geometrically split). There are a number of ways one could try 
to estimate the number of principally-polarized surfaces in these isogeny 
classes; for instance, the techniques of~\cite{Howe2004} might be of use.
We will not speculate further on this here.

Let $\cala_{2,\gsplit}(\fq)$ denote the subset of $\cala_{2}(\fq)$ consisting 
of those principally-polarized varieties that are not geometrically simple, and
for each $q$ let $d_q$ denote the ratio
\[
d_q = \frac{\sqrt{q} \cdot \#'\cala_{2,\gsplit}(\fq)}{\#'\cala_2(\fq)}
    = \frac{\sqrt{q} \cdot \#'\cala_{2,\gsplit}(\fq)}{q^3 + q^2}.
\]    
While collecting the data presented in Section~\ref{ssec-data-probsplit} we 
also collected data on $d_q$ by random sampling of curves. 
Figure~\ref{figprobgeomsplit}  presents the results for 
$q\in\st{131, 257, 521, 1031, 2053, 4099, 16411, 65537}.$  The figure suggests
that perhaps $d_q$ is bounded away from $0$ and~$\infty$.

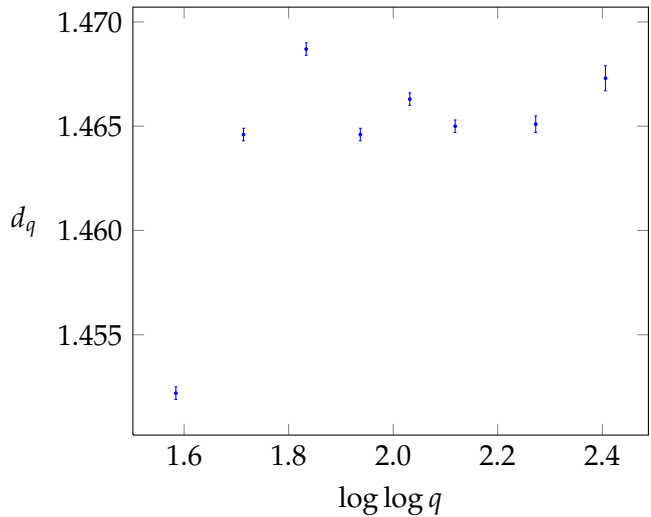
\begin{figure}
\centering
\begin{tikzpicture}
   \begin{axis}[xlabel=$\log \log q$,
                ylabel=$d_q$, 
                xtick = {1.6,1.8,2.0,2.2,2.4},
                xticklabels = {$1.6$,$1.8$,$2.0$,$2.2$,$2.4$},
                ytick = {1.450,1.455,1.460,1.465,1.470},
                yticklabels = {$1.450$,$1.455$,$1.460$,$1.465$,$1.470$},
                ylabel style={rotate=270},
                mark size={0.1ex}]
   \addplot[color=blue, only marks,
            error bars/.cd,
            y dir=both, y explicit,
             ]
     coordinates {
( 1.5842, 1.4522)  +- (0,0.0003)
( 1.7136, 1.4646)  +- (0,0.0003)
( 1.8335, 1.4687)  +- (0,0.0003)
( 1.9371, 1.4646)  +- (0,0.0003)
( 2.0317, 1.4663)  +- (0,0.0003)
( 2.1185, 1.4650)  +- (0,0.0003)
( 2.2727, 1.4651)  +- (0,0.0004)
( 2.4061, 1.4673)  +- (0,0.0006)
        };
  \end{axis}
\end{tikzpicture}
\caption{The experimentally computed values of $d_q$ for the primes $q$ 
         in $\st{131, 257, 521, 1031, 2053, 4099, 16411, 65537}$. 
         Error bars indicate one standard deviation.}
\label{figprobgeomsplit}
\end{figure}

\bibliographystyle{hplaindoi} 
\bibliography{splitsurf}

\end{document}